\documentclass[11pt]{article}

\usepackage{epsfig}
\usepackage{amsmath,subfigure}
\usepackage{amsthm}
\usepackage{amssymb}
\usepackage{color}
\usepackage[utf8]{inputenc}
\usepackage[english]{babel}

\newcommand{\footremember}[2]{%
   \footnote{#2}
    \newcounter{#1}
    \setcounter{#1}{\value{footnote}}%
}

\usepackage[titletoc,toc,title]{appendix}
\def \ve{\varepsilon}
\def \IR{\mathds{R}}
\newcommand \la{\langle}
\newcommand \ra{\rangle}

\usepackage[top=2 cm,bottom=2 cm,left=2. cm, right=2. cm]{geometry}

\usepackage{dsfont}
\newtheorem{theorem}{Theorem}[section]

\newtheorem{lemma}[theorem]{Lemma}

\newtheorem{definition}[theorem]{Definition}

\newtheorem{assumption}[theorem]{Assumption}

\newtheorem{remark}[theorem]{Remark}

\def\Xinttt#1{\mathchoice
{\XXinttt\displaystyle\textstyle{#1}}%
{\XXinttt\textstyle\scriptstyle{#1}}%
{\XXinttt\scriptstyle\scriptscriptstyle{#1}}%
{\XXinttt\scriptscriptstyle\scriptscriptstyle{#1}}%
\!\int}
\def\XXinttt#1#2#3{{\setbox0=\hbox{$#1{#2#3}{\int}$}
\vcenter{\hbox{$#2#3$}}\kern-.52\wd0}}
\def\dashint{\Xinttt-}

\title{Homogenization of  some degenerate pseudoparabolic variational inequalities }
\author{ Mariya Ptashnyk \footremember{Dnd}{Department of Mathematics, School of Mathematical and Computer Sciences, Heriot-Watt University, Edinburgh EH14 4AS, UK; Division of Mathematics, University of Dundee, Dundee, DD1 4HN, UK, m.ptashnyk@hw.ac.uk}
}

\begin{document}
\maketitle

\begin{abstract} 
Multiscale analysis of a degenerate pseudo\-parabolic variational inequality, modelling  the two-phase flow with dynamical capillary pressure in a perforated domain, is the main topic of this work.  Regularisation and penalty operator methods are applied to show the existence of a solution of the nonlinear  degenerate pseudoparabolic variational inequality defined in a domain with microscopic perforations, as well as to derive a priori estimates for  solutions of the microscopic problem. The main challenge   is  the derivation of  a priori estimates for solutions of the variational inequality, uniformly with respect to the regularisation parameter and to the small parameter defining the scale of the microstructure. The method of two-scale convergence is used to derive the corresponding macroscopic obstacle problem.
\end{abstract}

{\small {\it Keywords:}
 degenerate PDEs,  pseudoparabolic variational inequalities,  obstacle problems, penalty operator method,  two-scale convergence,  homogenization
}

\section{Introduction} 
In this paper we consider multiscale analysis of a nonlinear degenerate pseudoparabo\-lic variational  inequality modelling  unsaturated flow  with dynamic capillary pressure in a  perforated porous medium.    Models for two-phase flow with dynamical capillary pressure, originally   proposed by \cite{Hassanizadeh, Pavone_1989}, 
consider Darcy's law for the flux of the  moisture content $u$  given by
$$
J= - A\,  k(u)(\nabla p + {\bf e}_n), 
$$
 and  assume that the pressure $p$ in the wetting phase  is  a function of the moisture content $u$ and its time derivative $\partial_t u$, i.e.\ in a simplified form,
$$
p=- \tilde P_c(u) + \tau \partial_t u, 
$$
 where the permeability function  $k(u)$ depends on the moisture content, the vector ${\bf e}_n =(0, \ldots, 0,1)$ determines the   direction of flow  due to gravity, and 
$A$ and $\tau$ are positive constants.  Then for the moisture content   $u$ we obtain a pseudoparabolic equation of the from
\begin{equation}\label{eq1}
\partial_t u = \nabla\cdot \big(A\,  k(u) [P_c (u) \nabla u + \tau \nabla \partial_t u + {\bf e}_n]\big), 
\end{equation}  
where $P_c(u) =- \tilde P_c^\prime(u)$.

If   considering a two-phase flow problem  in a perforated porous medium   with  Signorini's type conditions on the surfaces of perforations
\begin{equation}\label{eq2}
\begin{aligned}
& u\geq 0, \; \; \;  A\,   k(u) (P_c (u) \nabla u + \tau \nabla \partial_t u + {\bf e}_n)\cdot \nu\geq - f(t,x,u), \\
& u\big[A \,  k(u) (P_c (u) \nabla u + \tau \nabla \partial_t u + {\bf e}_n)\cdot \nu  + f(t,x,u)\big] =0, 
\end{aligned}
\end{equation} 
then  a weak formulation of  equation \eqref{eq1} together with  conditions  \eqref{eq2} results in a pseudoparabolic variational inequality of the form
\begin{equation}\label{ineq_11}
\begin{aligned} 
 \langle \partial_t u, v - u \rangle_{G^\ve}   +\left \langle A\,   k(u) [ P_c(u) \nabla u + \tau  \partial_t \nabla u +  {\bf e}_n] , \nabla(v - u)\right \rangle_{G^\ve}    + 
\langle  f(t,x,u), v - u \rangle_{\Gamma^\ve}  \geq 0, 
\end{aligned} 
\end{equation}
where $G^\ve\subset \mathbb R^n$, with $n=2,3$, denotes the perforated domain  and $\Gamma^\ve$ defines the boundaries of perforations.  

As an example of a porous medium with microscopic perforations we can consider a part of the soil perforated  by a root network, where  conditions \eqref{eq2} model water (solute) uptake by plant roots. 

In our analysis of the  obstacle problem \eqref{eq1} and  \eqref{eq2}, or equivalently variational inequality \eqref{ineq_11},   defined in a heterogeneous perforated domain $G^\ve$, where $\ve$ denotes a characteristic  size of perfora\-tions,  we shall consider  a function  $A(x)$ describing  the heterogeneity of the medium, instead of a constant A, and  a more general convection term,  describing flow transport by  a given velocity field.

Along with models for two-phase flow with dynamic capillary pressure \cite{Cuesta, Hassanizadeh, Pavone_1989},   pseudoparabolic equations are also used to model fluid filtration in  fissured porous media \cite{Barenblatt_1990}, heat transfer in a hete\-rogeneous medium \cite{Rubinshtein},  or  to regularise ill-posed transport problems \cite{Barenblatt_1993, Novick}.  Pseudoparabolic variational inequalities are considered  to describe  obstacle  \cite{Scarpini_1987} and free boundary problems \cite{DiBenedetto}. 
The well-posedness for non-degenerate pseudoparabolic equations and variational ine\-qua\-lities was studied  by many authors  \cite{Boehm_1985,  Pop,  DiBenedetto, Kenneth_1984, Mikelic,  Ptashnyk_1, Ptashnyk_2, Scarpini_1987, Showalter}.    Global existence results for  degenerate pseudoparabolic equations   are obtained  in \cite{Pop, Mikelic}.  The multiscale analysis  for non-degenerate pseudoparabolic equations was considered  in  \cite{Peszynska} and the method of  two-scale convergence was applied to derive the corresponding macroscopic equations.  To the best of  our knowledge, there are no  results  on homogenization of   pseudopara\-bo\-lic variational inequalities.    Several results are known on multiscale analysis of  ellip\-tic   \cite{Timofte_2016, Iosifyan, Jaeger_2014,  Pastukhova_2001, Sandrakov, Vorobev_2003}  and parabolic \cite{Jaeger, Mielke_Timofte, Rodrigues_1982, Shaposhnikova_2008} variational inequalities. 
In \cite{Timofte_2016} the periodic unfolding method was used to derive macroscopic variational inequality for the microscopic Signorini-Tresca problem.  
The method of  two-scale convergence was applied to derive macroscopic problems for  microscopic linear elasticity  equations with boundary conditions of Signorini types \cite{Iosifyan},  elliptic variational inequalities for obstacle problems \cite{Sandrakov},  and  evolutionary variational inequalities  \cite{Mielke_Timofte}. Weak convergence and  construction of  a corrector were considered  in \cite{Jaeger_2014, Shaposhnikova_2008, Vorobev_2003} to derive macroscopic problems for microscopic elliptic and parabolic variational inequalities  under certain conditions on the relation between the period  and the size of the microstructure.     In \cite{Rodrigues_1982} the multiscale analysis of a  parabolic variation inequality corresponding to the Stefan problem was performed using the H-convergence method \cite{Murat_1997}.    
Homogenization of variational inequalities in  domains with thick junctions, for which standard extension results do not hold, was studied  in \cite{Melnyk_2011, Melnyk_2016, Melnyk_2012} using the method of monotone operators  and construction of  appropriate auxiliary functions.

To prove existence of a solution of the microscopic problem, considered here,   the regularisation of degenerate coefficients in the pseudoparabolic variational inequality  together with a proper choice of test functions, similar to those proposed in \cite{Pop, Mikelic} for pseudoparabolic equations,  is considered. In the case of variational inequalities additional care is required  due to the fact that admissible test functions have to belong to a convex subset of the corresponding function space.  The penalty operator method is applied to show existence of a solution of the pseudoparabolic variational inequality with regularised coefficients.  To pass to the limit in the nonlinear penalty operator we prove  strong convergence of  approximations of solutions of the corresponding nonlinear pseudoparabolic equation.  
The main step in the analysis and derivation of the  macroscopic variational inequality, for the microscopic  problem  considered here,  is to derive a priori estimates uniformly with respect to small parameter $\ve$.   The main idea in the derivation of a priori estimates for the time derivative of the gradient of a solution of variational inequality, similar to \cite{Mikelic},  is to use the specific structure of the degenerate coefficients which allows to prove that some negative power of a solution of the variational inequality is a $L^p$-function with $1<p<2$.   The uniqueness result is obtained   in the case when the coefficient $k(u)$ in  front of the pseudoparabolic term is non-degenerate  and under  additional regularity assumptions  on solutions of the pseudoparabolic variational inequality.

The paper is organised as follows.  In Section~\ref{section1} we formulate the microscopic  obstacle problem defined in a perforated domain $G^\ve$. In Section~\ref{section12} we  prove existence and uniqueness results for the  regularised problem,  derive a priori estimates,  and show existence of a solution  of the original  degenerate   pseudo\-parabolic  variational inequality defined in the perforated domain~$G^\ve$. In Section~\ref{section2} we prove convergence results as $\ve \to 0$ and derive macroscopic problem defined in a homogeneous domain $G$ with the  constraint $u(t,x) \geq 0$ in $(0,T)\times G$. In Appendix we summarise the main compactness results for the two-scale convergence used in the derivation of the macroscopic pseudoparabolic variational inequality.

\section{Formulation of mathematical problem} \label{section1}
A general obstacle problem can be formulated as a variational inequality  
\begin{equation}\label{main1}
\begin{aligned} 
& u \in \mathcal K(t) , \\
& \langle \partial_t b(u), v - u \rangle  + \langle \mathcal A(x, \nabla u, \partial_t \nabla u), \nabla (v - u) \rangle \geq \langle R(t,x, u), v - u \rangle
\end{aligned} 
\end{equation}
for $v \in L^2(0,T; \mathcal K(t) )$, where  $\mathcal K(t)$ is a closed convex set in $H^1(G)$. We shall consider  variational inequality \eqref{main1} defined in a perforated domain $G^\ve$   with a periodic distribution of perforations.  

To define the domain $G^\ve$, where  $\ve$ denotes the characteristic  size of perforations, we consider a bounded  domain $G\subset \mathbb R^n$, for $n=2, 3$,  where $G$ is quasi-convex or   $\partial G\subset C^{1,\alpha}$ for some $0<\alpha <1$, a `unit cell' $Y\subset \mathbb R^n$, a subset $Y^0$,  with   $\overline{Y^0} \subset Y$ and Lipschitz boundary $\Gamma = \partial Y^0$, and denote $Y^\ast = Y \setminus \overline{Y^0}$. Then    
$$
G^\ve_0 = \bigcup_{\xi\in \Xi^\ve} \ve(Y^0 + \xi),  \; \;  \widetilde G^\ve ={\rm Int}  \bigcup_{\xi\in \Xi^\ve} \ve(\overline Y + \xi),
$$
where $\Xi^\ve= \{ \xi  \in \mathbb Z^n\; : \;  \ve(\overline{Y^0} + \xi)  \subset G \}$, 
and  $G^\ve = G\setminus \overline {G^\ve_0}$.  The   boundaries of perforations are defined by 
$$
\Gamma^\ve = \bigcup_{\xi\in \Xi^\ve} \ve(\Gamma + \xi) .
$$
For the nonlinear function $\mathcal A$ in the variational inequality in \eqref{main1}  we consider   
$$\mathcal A(x, \nabla u^\ve, \partial_t \nabla u^\ve) = A^\ve(x) k(u^\ve) (P_c(u^\ve)  \nabla u^\ve +  \partial_t \nabla  u^\ve ) - F^\ve(t,x, u^\ve), $$
and assume that  $R(t,x,u^\ve)=0$, where the functions  $b$, $A^\ve$, $k$, $P_c$, and $F^\ve$ are specified below. 
 On the  microscopic boundaries $\Gamma^\ve$ we specify the following Signorini type conditions 
\begin{equation*}
\begin{aligned} 
   u^\ve & \geq 0, \quad \;  \\
 \big( A^\ve(x) k(u^\ve) [ P_c(u^\ve) \nabla u^\ve +  \partial_t \nabla u^\ve ] - F^\ve(t,x, u^\ve)\big)  \cdot \nu   + \ve f^\ve(t,x, u^\ve) & \geq 0,  \\
 u^\ve   \left[\big( A^\ve(x) k(u^\ve)  [P_c(u^\ve) \nabla  u^\ve +  \partial_t \nabla u^\ve ] - F^\ve(t,x, u^\ve)\big) \cdot \nu + \ve f^\ve(t,x,u^\ve)\right] & =0, 
\end{aligned} 
\end{equation*} 
where  function $f^\ve$ is specified below. 
Then the closed convex set  $\mathcal K^\ve$ is defined as 
\begin{equation}\label{K}
\mathcal K^\ve = \{ v \in H^1(G^\ve)\;  : \;   v = \kappa_D \; \text{ on } \partial G,  \;  v \geq  0 \text{ on } \Gamma^\ve \}, 
\end{equation}
with some constant   $0<\kappa_D \leq 1$,   
  and   the corresponding  variational inequality reads
\begin{equation}\label{main2}
\begin{aligned} 
 \langle \partial_t b(u^\ve), v - u^\ve \rangle_{G^\ve_T}  +\left \langle A^\ve(x) k(u^\ve) [ P_c(u^\ve) \nabla u^\ve +  \partial_t \nabla u^\ve] , \nabla(v - u^\ve)\right \rangle_{G^\ve_T} \qquad  \\  
 \qquad -
\langle   F^\ve(t,x, u^\ve) ,  \nabla (v - u^\ve) \rangle_{G^\ve_T}   + 
\langle \ve f^\ve(t,x,u^\ve), v - u^\ve \rangle_{\Gamma^\ve_T}  \geq 0,
\end{aligned} 
\end{equation}
for $v - \kappa_D\in L^2(0,T; V)$ and  $v(t)\in  \mathcal K^\ve$ for $t\in (0,T)$, where 
$$
V= \{ v \in H^1(G^\ve) : \; v = 0 \; \text{ on } \partial G \}. 
$$

\noindent  Here we use notation  $G_T = (0,T)\times G $, $G^\ve_T= (0,T)\times G^\ve$, $\Gamma_T= (0,T)\times \Gamma$,  $\Gamma^\ve_T = (0,T)\times \Gamma^\ve$,  $Y_T = (0,T)\times Y$,  $Y^\ast_T = (0,T)\times Y^\ast$, and
$$
\begin{aligned} 
\langle \phi, \psi \rangle_{G^\ve_T} =\int_0^T  \int_{G^\ve} \phi\,   \psi  \, dxdt, \quad  \text{  for } \phi \in L^p(0,T; L^q(G^\ve)), \;   \psi \in L^{p^\prime}(0,T; L^{q^\prime}(G^\ve)),
\\  
\langle \phi, \psi \rangle_{\Gamma^\ve_T} = \int_0^T \int_{\Gamma^\ve} \phi \,  \psi  \, d\gamma dt,  \quad   \text{ for } \phi \in L^p(0,T; L^q(\Gamma^\ve)), \;  \psi \in L^{p^\prime}(0,T; L^{q^\prime}(\Gamma^\ve)), 
\end{aligned}
$$
where $1< p, p^\prime, q, q^\prime  < \infty$ with  $1/p + 1/p^\prime =1$ and $1/q + 1/q^\prime =1$.\\
{\bf Remark. } Notice that $\langle \cdot, \cdot \rangle_{G^\ve_T}$ and $\langle \cdot, \cdot \rangle_{\Gamma^\ve_T}$ are used as short notation for an integral  of a product of two functions. In most cases we will consider a product of two $L^2$-functions, however we shall use the same notation for the integral of a product of $L^p$- and $L^{p^\prime}$-functions, which is well defined.  \\

We shall consider the following assumptions on functions  $A^\ve$, $b$, $k$, $P_c$, $F^\ve$,  and $f^\ve$.
\begin{assumption}\label{assum_1}
\begin{itemize}    
\item[1)]   $k : \mathbb R \to \mathbb R $ is Lipschitz continuous, nondecreasing,  with  $k(z) >0$  for  $ z>0$ and  $ k(0) = 0$,  e.g.\    
$$k(z) = \frac{\vartheta_k z^\beta}{1+ \gamma_k  z^\beta}  \quad 
\text{  for  some } \,  \vartheta_k, \,    \gamma_k  >0 \text{  and  }  \beta \geq 1, $$ 
$$ P_c(z) = \frac{\vartheta_p z^{-\lambda}}{1+\gamma_p(z) z^\lambda}  \; \; \text{  for    }  \; \vartheta_p,  \lambda>0, \text{ nonnegative } \gamma_p \in C^\infty_0(\mathbb R),   \text{ and  }
 |k(z) P_c(z)| \leq C < \infty  \; \text{   for }  z\geq 0. $$   

\item[2)] $A\in L^\infty(Y)$ is extended $Y$-periodically to $\mathbb R^n$, and $A(y) \geq a_0 >0$ for $y \in Y$,   with    $A^\ve(x) = A(x/\ve)$ for $x\in \mathbb R^n$. 

 \item[3)]   $ b : \mathbb R \to \mathbb R$ is  continuous,  nondescreasing,  and twice continuously differentiable for $z>0$, with $b(z) > 0 $ for   $z > 0$,  $b(0) = 0$,   and $|b^\prime(z)| \leq \gamma_b (1+ z^2)$ for $z\geq 1$ and  $\gamma_b>0$, e.g.\   
 $b(z) = \vartheta_b z^\alpha$, with $0<\alpha \leq 3$ and $\vartheta_b >0$.
\item[4)]    $F^\ve:  \mathbb R^+  \times \mathbb R^{n}\times \mathbb R  \to \mathbb R^n$  is Lipschitz continuous,  
$F^\ve(t,x,z) = Q^\ve(t,x) H(z) + k(z) g$,     where   $ |H^\prime(z)(b^\prime(z))^{-\frac 12} | \leq C< \infty$  for $ z \geq 0$, 
$ g \in \mathbb R^n$ is a constant vector,  $\nabla_x \cdot Q^\ve(t, x) = 0 $  for $ (t,x) \in G^\ve_T$,   $ Q^\ve(t,x)\cdot \nu = 0 $  on $ \Gamma^\ve_T$,  $Q^\ve \in L^\infty(G_T^\ve)$, 
and  $Q^\ve(t,x) \to Q(t,x,y)$ strongly two-scale, $Q \in L^2(G_T, H_{\rm div}(Y^\ast))\cap L^\infty(G_T\times Y^\ast)$, where 
$  H_{\rm div}(Y^\ast) = \{ v \in L^2(Y^\ast)^n, \; \nabla_y\cdot v=0 \text{ in } Y^\ast,  \; \text{ and} \;    v  \text{ is }  Y\text{-periodic}\}. $
\item[5)]  $ f^\ve(t,x,\xi) = f_0(t,x/\ve) f_1(\xi)$, where  $ f_0 \in C^1([0,T]; C^1_{\rm per} (\Gamma))$,  with  $ f_0(t,y) \geq 0  $  for $ (t,y) \in  \Gamma_T$, and 
 $f_1\in C^1_0(\mathbb R)$,  with  $ \xi f_1(\xi) \geq 0$,   $f_1(0) = 0$,  and 
 $$ \Big |f_1(\xi) \int_\xi^{\kappa_D} \frac {d\eta} {k(\eta)} \Big|\leq C \; \;  \text{ for } \;  0 \leq \xi \leq \kappa_D. $$   
\item[6)] Initial condition $u_0\in \mathcal K$   and 
$$ \int_{\kappa_D}^{u_0} b^\prime(\xi) \int_{\kappa_D}^{\xi} \frac {dz} {k(z)}  \, d \xi \;   \in\,  L^1(G), $$
\begin{equation}\label{K_macro}
\text{ where } \; \; \mathcal K = \{ v \in H^1(G) : \;   v = \kappa_D \; \text{ on } \partial G  \; \text{ and } \;   v \geq  0 \text{ in } G\}.
\end{equation}
\end{itemize} 
\end{assumption} 
\noindent{\bf Remark.} Notice that assumptions 1) and 4)  in Assumption~\ref{assum_1} are similar to the corresponding assumptions in \cite{Pop, Mikelic}, however for the vector field $Q^\ve$ additional assumptions are required due to the perforated microstructure of domain $G^\ve$.  Function $F^\ve$ describes the directed flow due to   a given velocity field $Q^\ve$ and gravity $g$. 
As an example of a function $Q^\ve$ satisfying assumption 4) we can consider  a solution of the Stokes problem 
\begin{equation}\label{stokes}
\begin{aligned} 
- \mu\,  \ve^2  \Delta Q^\ve + \nabla p^\ve = 0 \quad \text{ in } \; G^\ve, && \qquad 
{\rm div }\,  Q^\ve & = 0 \quad \text{ in } \; G^\ve, \\
Q^\ve = 0   \quad \text{ on } \;  \Gamma^\ve,  &&\quad 
Q^\ve &=  v     \quad \text{ on }\;  \partial  G, 
\end{aligned} 
\end{equation} 
for  $t \in (0,T)$ and a  given velocity    $v \in L^\infty(0,T; H^{2} (G))^n$ with ${\rm div } \, v(t,x)  = 0$  for  $x\in G$ and $t\in (0,T)$.  The regularity theory for Stokes equations, see e.g.\ \cite{Brown_Shen, Geng_Kilty, Mitrea_Wright},  implies that for each fixed $\ve$ there exists a solution $(Q^\ve, p^\ve) \in L^\infty(0,T; W^{1,p}(G^\ve)) \times L^\infty(0,T; L^p(G^\ve) / \IR)$, with  $2\leq p < n + \delta_1$ and some $\delta_1 >0$,   of system \eqref{stokes}.  Then using the Sobolev embedding theorem we obtain    $Q^\ve \in L^\infty(0,T; L^\infty(G^\ve))$. 
The multiscale analysis results for the Stokes system, see e.g.\ \cite{Hornung},  imply existence of a velocity field $Q \in L^\infty(0,T; L^2(G; H^1_{\rm per} (Y^\ast)))$, pressure $p \in L^\infty(0,T; L^2(G)/\mathbb R)$,   and $\pi \in L^\infty(0,T; L^2(G\times Y^\ast)/\mathbb R)$,  such that   $Q^\ve \rightharpoonup Q$ two-scale,  $p^\ve \rightharpoonup p$ weakly-$\ast$ in $L^\infty(0,T; L^2(G))$,  and  $Q$ is a solution of
\begin{equation}\label{Stokes_micro-Macro}
\begin{aligned} 
 -& \mu\,  \Delta_y Q +  \nabla_y \pi + \nabla p  = 0 \quad && \text{ in } \;  Y^\ast, \\
& {\rm div}_y\,  Q = 0 \quad \text{ in } \;  Y^\ast, \quad 
Q = 0   \; \;  && \text{ on } \;   \Gamma,    
\end{aligned} 
\end{equation}
and  ${\rm div} \int_{Y^\ast} Q(t, x,y) dy = 0$  for $(t,x) \in G_T$, with  
$$
\begin{aligned} 
{\rm div}( K \nabla p  ) =  0 \text{ in } G, \qquad  K \nabla p \cdot \nu = v \cdot \nu \; \; \text{ on } \partial G,   
\end{aligned} 
$$
for $t\in [0, T]$ and constant permeability tensor $K$ determined by the corresponding `unit cell' problems.   Using the regularity theory for  elliptic equations with Neumann boundary conditions,  together with the assumptions on $G$ and $v$,  we obtain 
$\nabla p \in L^\infty(0,T; L^\infty(G))^n$, see e.g.\  \cite{Behrndt_Micheler,  Gesztesy_Mitrea, Kenig}.  Then  applying the regularity results for the Stokes system, see e.g.\ \cite{Brown_Shen, Geng_Kilty, Mitrea_Wright}, to  problem   \eqref{Stokes_micro-Macro} yields $Q \in L^\infty(G_T; L^\infty(Y^\ast))$. Notice that in \eqref{Stokes_micro-Macro} variables $t$ and $x$ play the role of parameters in the Stokes operator  with respect to the microscopic variable  $y$.

To show strong two-scale convergence of $Q^\ve$ we consider $ Q^\ve - R^\ve_{Y^\ast} (v)$, where  $ R^\ve_{Y^\ast} (v)(x) =  R_{Y^\ast} (v_\xi^\ve)(x/\ve)$, with  $v_\xi^\ve(y) = v(\ve y)$ for $y \in \ve(Y + \xi)$ and
$\xi \in \Xi^\ve$,   and $R_{Y^\ast}: W^{1,p}(Y)^n \to W^{1,p}_\Gamma(Y^\ast)^n$,  for $1< p < \infty$,  is a  restriction operator,  see e.g.\ \cite{Mikelic_1991, Tartar},      as a test function in~\eqref{stokes}  and obtain
\begin{equation}\label{Q_limit}
\begin{aligned}
 \mu \|\nabla_y Q\|^2_{L^2(G\times Y^\ast)}   \leq \mu \liminf_{\ve \to 0 } \|\ve \nabla Q^\ve\|^2_{L^2(G^\ve)} 
\leq \mu   \limsup_{\ve \to 0 } \|\ve \nabla Q^\ve\|^2_{L^2(G^\ve)} 
=\lim\limits_{\ve \to 0 }  \ve^2  \mu \la \nabla Q^\ve, \nabla R^\ve_{Y^\ast} (v) \ra_{G^\ve}   
\end{aligned} 
\end{equation}
for $t\in [0,T]$.   Here $W^{1,p}_\Gamma(Y^\ast)^n = \{ w \in W^{1,p}(Y^\ast)^n\, : \, w =0 \text{ on } \Gamma\}$.   Notice that $R^\ve_{Y^\ast} (v) = v$ in $G \setminus  \widetilde G^\ve$ and  the construction of the restriction operator ensures
$$
\begin{aligned} 
\| {\mathcal T}^{\ve_n}  R^{\ve_n}_{Y^\ast}(v) -  {\mathcal T}^{\ve_m} R^{\ve_m}_{Y^\ast} (v) \|_{L^2(G \times Y^\ast)} +  \| \nabla_y ({\mathcal T}^{\ve_n} R^{\ve_n}_{Y^\ast}(v) - {\mathcal T}^{\ve_m}  R^{\ve_m}_{Y^\ast} (v) ) \|_{L^2(G \times Y^\ast)}
\\ \leq C  
\big[\| {\mathcal T}^{\ve_n} v  -  {\mathcal T}^{\ve_m}v \|_{L^2(G \times Y)} +  \|\nabla_y  {\mathcal T}^{\ve_n} v  -  \nabla_y {\mathcal T}^{\ve_m}v \|_{L^2(G \times Y)}  \big] \to 0 
\end{aligned} 
$$ 
as $n, m \to  \infty$ and for $t \in [0,T]$, where $\mathcal T^\ve$ is the periodic unfolding operator, see e.g.\ \cite{Cioranescu}.  Hence $R^{\ve}_{Y^\ast}(v) \to \hat R_{Y^\ast} (v)$  and  $\ve \nabla R^{\ve}_{Y^\ast}(v) \to  \nabla_y \hat R_{Y^\ast} (v)$   strongly two-scale as $\ve \to 0$, with $\hat R_{Y^\ast} (v) \in L^\infty(G_T; H^1_{\rm per} (Y^\ast))$. 
Then  using the two-scale convergence of $Q^\ve$ we obtain  
\begin{equation}\label{Q_lim_2}
\lim\limits_{\ve \to 0 }  \ve^2  \mu \la \nabla Q^\ve, \nabla R^\ve_{Y^\ast} (v) \ra_{G^\ve}  =  \mu  \la \nabla_y  Q,  \nabla_y \hat R_{Y^\ast} (v)   \ra_{L^2(G\times Y^\ast)}  
\end{equation}
for $t \in [0, T]$. Taking $Q- \hat R_{Y^\ast} (v)$ as a test function in \eqref{Stokes_micro-Macro} and using the fact that $\hat R_{Y^\ast} (v)(t, x, \cdot )$ is $Y$-periodic,  $\hat R_{Y^\ast} (v) =0$ on $\Gamma$,   ${\rm div}_y \hat R_{Y^\ast} (v) =0$,  and ${\rm div} \int_{Y^\ast} \hat R_{Y^\ast} (v)(t, x,y) dy =0$,  yield  
$$
\mu\,  \la \nabla_y Q, \nabla_yQ- \nabla_y \hat R_{Y^\ast} (v)  \ra_{G\times Y^\ast}  = 0 
$$
for $t \in [0,T]$. Combining the last equality with inequality  \eqref{Q_limit}  and convergence in  \eqref{Q_lim_2}  implies 
$$
\lim_{\ve \to 0 } \|\ve \nabla Q^\ve\|_{L^2(G^\ve)} = \|\nabla_y Q\|_{L^2(G\times Y^\ast)}
$$
for $t\in[0,T]$, and we have the strong two-scale convergence of $\ve \nabla Q^\ve$ and strong convergence of  unfolded sequence   $\nabla_y \mathcal T^\ve Q^\ve$   in $L^2(G_T\times Y^\ast)$.  Using zero Dirichlet boundary conditions on $\Gamma^\ve$  and applying the Poincare inequality we obtain  
$$
\|\mathcal T^{\ve_m} Q^{\ve_m} -  \mathcal T^{\ve_n} Q^{\ve_n}\|_{L^2(G_T\times Y^\ast)} \leq   C \|\nabla_y (\mathcal T^{\ve_m}  Q^{\ve_m}  - \mathcal T^{\ve_n} Q^{\ve_n}) \|_{L^2(G_T\times Y^\ast)} \to 0 
$$
as $n, m \to \infty$. Thus we have strong convergence  of  $\mathcal T^{\ve} Q^{\ve}$  in $L^2(G_T \times Y^\ast)$  and strong   two-scale convergence of $Q^\ve$ to  $Q$. \\
 
As next we give the definition of a solution of the microscopic inequality  \eqref{main2}. 
\begin{definition} 
A solution of  inequality    \eqref{main2} is  a  function $u^\ve$ such that  $u^\ve  -   \kappa_D \in  L^2(0, T ; V)$,    $\partial_t b(u^\ve) \in L^2(0,T; L^{r}(G^\ve))$, 
with  $6/5\leq r < 4/3$,    
$\sqrt{k(u^\ve)} \nabla \partial_t u^\ve \in L^2(G_T^\ve)$,  and  $u^\ve(t) \in \mathcal K^\ve$ for $t\in[0, T]$, and $u^\ve$ satisfies   variational inequality \eqref{main2} 
for $v \in  L^2(0,T; \mathcal K^\ve)$  and initial condition $u^\ve(t) \to u_0$ in $L^2(G^\ve)$ as $t \to 0$. 
\end{definition}

\section{A priori estimates and existence result} \label{section12}

Similar to \cite{Mikelic}, in order to prove the existence result for  variational inequality \eqref{main2},   we first consider  regularisation of functions $b$, $k$, and $P_c$, given  by $b_\delta(v) = b(v^{+} + \delta)$, with  $b_\delta(v) = b(v)$ if $b(v)=\vartheta_b v$ for some constant $\vartheta_b>0$,    $k_\delta(v) = k(v^{+}+ \delta)$, and  $P_{c, \delta}(v) = P_c(v^{+} + \delta)$,   where   $\delta >0$ and $v^{+} = \max\{ v, 0 \}$. 

Then the  corresponding regularised problem reads
 \begin{eqnarray}\label{main2_regul}
  \langle \partial_t b_\delta(u^\ve_\delta), v - u^\ve_\delta \rangle_{G^\ve_T}  +  \langle A^\ve(x) k_\delta(u^\ve_\delta) [P_{c, \delta}(u^\ve_\delta) \nabla u^\ve_\delta +  \partial_t \nabla u^\ve_\delta] , \nabla(v - u^\ve_\delta)  \rangle_{G^\ve_T} \nonumber \\
 - \langle   F^\ve(t,x, u^\ve_\delta) ,  \nabla (v - u^\ve_\delta) \rangle_{G^\ve_T}   + \langle\ve f^\ve(t,x, u^\ve_\delta), v - u^\ve_\delta \rangle_{\Gamma^\ve_T}     \geq 0, \;  \quad   \\
 \text{and }  u^\ve_\delta(t) \in \mathcal K^\ve, \;   \text{ for } v \in L^2(0, T;  \mathcal K^\ve),  \hspace{4.1 cm }  \nonumber
\end{eqnarray}
and $u^\ve_\delta(0) = u_0$ in $L^2$-sense. 

To show the existence of a solution of problem  \eqref{main2_regul} we  apply the penalty operator method \cite{Kinderlehrer, Lions}  and consider 
\begin{equation}\label{penalty}
\begin{aligned} 
\partial_t b_\delta( u^\ve_{\delta, \mu}) - \nabla \cdot \big( A^\ve(x) k_\delta(u^\ve_{\delta, \mu}) [ P_{c, \delta}(u^\ve_{\delta, \mu}) \nabla u_{\delta, \mu}^\ve +  \partial_t  \nabla u^\ve_{\delta, \mu}  ] \big)  + 
\nabla \cdot F^\ve(t,x, u^\ve_{\delta, \mu}) \\
+ \frac 1 \mu \mathcal B (u^\ve_{\delta, \mu} - \kappa_D) = 0  \quad  & \text{ in } G^\ve_T,  \\
\big(A^\ve(x) k_\delta(u^\ve_{\delta, \mu})  [P_{c,\delta}(u^\ve_{\delta, \mu}) \nabla  u^\ve_{\delta, \mu} +  \partial_t \nabla u^\ve_{\delta, \mu} ] - F^\ve(t,x, u^\ve_{\delta, \mu})\big) \cdot \nu  = -  \ve f^\ve(t,x,u^\ve_{\delta, \mu})   \quad  &\text{ on }  \Gamma^\ve_T, 
\end{aligned} 
\end{equation} 
where $\mu >0$ and  a penalty operator  $\mathcal B: L^2(0,T; V) \to L^2(0,T; V^\prime)$ is monotone, bounded, hemicontinuous, and  $\mathcal B(v-\kappa_D) =0$ for $v(t) \in \mathcal K^\ve$.

\begin{lemma} \label{Lemma_existence_regular}
Under Assumption~\ref{assum_1}   there exists a  solution  $u_\delta^\ve \in L^2(0,T; \mathcal K^\ve)$  of  \eqref{main2_regul}  completed  with  initial condition $u^\ve_{\delta}(0) = u_0$ in $G^\ve$, with $\partial_t u_\delta^\ve \in L^2(0,T; H^1(G^\ve))$ and $\partial_t b_\delta(u^\ve_\delta) \in L^2(G^\ve_T)$. Under additional regularity assumption  $\partial_t u^\ve_\delta \in L^2(0,T; W^{1, p}(G^\ve))$  and $u_0 \in W^{1, p}(G^\ve)$  for $p>n$,  or if $k(\xi)= {\rm const}$,   $P_c$ is Lipschitz continuous for $\xi>0$, and $\nabla u_\delta^\ve \in L^2(0,T; L^p(G^\ve))$, variational inequality \eqref{main2_regul} has a unique solution. 
\end{lemma} 
\begin{proof}
First we shall apply   the Rothe and Galerkin methods to show existence of a weak solution of  \eqref{penalty}. Then by letting $\mu \to 0$ we will obtain  the existence result for variational inequality \eqref{main2_regul}. 
The  discretisation in time of equations in  \eqref{penalty}  yields the following elliptic problem for $ u^{\ve,j}_{\delta, \mu}(x) := u^{\ve}_{\delta, \mu}(t_j,x)$, for $x\in G^\ve$, 
\begin{equation}\label{penalty_time_discrete}
\begin{aligned} 
b^\prime_\delta(u^{\ve, j}_{\delta, \mu}) \frac 1 h( u^{\ve,j}_{\delta, \mu}  -  u^{\ve, j-1}_{\delta, \mu}) - \nabla \cdot \big( A^\ve(x) k_\delta(u^{\ve, j}_{\delta, \mu}) [ P_{c, \delta}(u^{\ve, j}_{\delta, \mu}) \nabla u^{\ve, j}_{\delta, \mu} +   \frac 1h  \nabla (u^{\ve, j}_{\delta, \mu}  - u^{\ve, j-1}_{\delta, \mu}) ] \big)  \\ + 
\nabla \cdot F^\ve(t_j,x, u^{\ve, j-1}_{\delta, \mu}) 
+ \frac 1 \mu \mathcal B (u^{\ve, j}_{\delta, \mu} - \kappa_D) = 0  \quad  & \text{ in } G^\ve,  \\
\Big(A^\ve(x) k_\delta(u^{\ve, j}_{\delta, \mu})  \big[P_{c,\delta}(u^{\ve, j}_{\delta, \mu}) \nabla  u^{\ve, j}_{\delta, \mu} +  \frac 1 h \nabla (u^{\ve, j}_{\delta, \mu} - u^{\ve, j-1}_{\delta, \mu}) \big] - F^\ve(t_j,x, u^{\ve, j-1}_{\delta, \mu})\Big) \cdot \nu \\   = -  \ve f^\ve(t_j,x,u^{\ve, j-1}_{\delta, \mu})   \quad  &\text{ on }  \Gamma^\ve, \\
u^{\ve, j}_{\delta, \mu} = \kappa_D  \quad &\text{ on }  \partial G, 
\end{aligned} 
\end{equation}
where  $h =T/N$ and $t_j = jh$,  for $j=1, \ldots, N$ and  $N \in \mathbb N$,  and $u^{\ve, 0}_{\delta, \mu}(x) = u_0(x)$ for $x\in G^\ve$. Since in this proof we assume that $\delta$ and $\ve$ are fixed, for the clarity of presentation we shall omit indices $\delta$ and $\ve$ in the calculations below. 
Now applying the Galerkin method to \eqref{penalty_time_discrete},  we   consider the  orthogonal system of basis functions  $\{\psi_i\}_{i \in \mathbb N}$ of the space $V$  and are looking for functions 
 $$
 u^j_{\mu, m}(x) = \kappa_D + \sum_{i=1}^m   \alpha_{mi}^j \psi_i(x) 
 $$
in the subspace $V_m= {\rm span} \{ \psi_1, \ldots, \psi_m \}$ such that 
 \begin{equation}\label{penalty_time_discrete_2}
\begin{aligned} 
 \la  b_\delta^\prime(u^{j}_{\mu, m})  \frac 1 h ( u^j_{\mu, m} - u^{j-1}_{\mu, m}), \zeta \ra_{G^\ve} 
  + \la  A^\ve(x) k_\delta(u^{j}_{\mu, m}) [ P_{c, \delta}(u^{j}_{\mu, m}) \nabla u^j_{\mu, m} +   \frac 1h  \nabla (u^j_{\mu, m}  - u^{j-1}_{\mu, m}) ], \nabla \zeta \ra_{G^\ve}  \\ 
  -  \la F^\ve(t_j,x, u^{j-1}_{\mu, m}), \nabla \zeta \ra_{G^\ve} 
+ \frac 1 \mu \la  \mathcal B (u^{j}_{\mu, m} - \kappa_D), \zeta \ra_{V^\prime, V} = -   \la \ve f^\ve(t_j,x,u^{j-1}_{\mu, m}) ,  \zeta \ra_{\Gamma^\ve}
\end{aligned} 
\end{equation}
for all  functions $\zeta \in V_m $.  Here $u^0_{\mu, m}$, with  $u^0_{\mu, m} - \kappa_D   \in  V_m$ and $u^0_{\mu, m} \in \mathcal K^\ve$, is a finite-dimensional  approximation of $u_0$.   
Thus we have  a system of algebraic equations for unknown coefficients $\alpha = ( \alpha_{m1}^j, \ldots,  \alpha_{mm}^j)$ and 
$$
\begin{aligned} 
 J(\alpha) \alpha =  
 \big\la b_\delta^\prime(v+\kappa_D)  \frac 1 h ( v + \kappa_D - u^{j-1}_{\mu, m}),  v + \kappa_D \big\ra_{G^\ve} 
 -  \big\la b_\delta^\prime(v+\kappa_D)  \frac 1 h ( v + \kappa_D - u^{j-1}_{\mu, m}),   \kappa_D \big\ra_{G^\ve} \\
  + \big\la  A^\ve(x) k_\delta(v+\kappa_D) \big[ P_{c, \delta}(v+\kappa_D) \nabla v +   \frac 1h  \nabla (v  - u^{j-1}_{\mu, m}) \big], \nabla v\ra_{G^\ve}   
  -  \la F^\ve(t_j,x, u^{j-1}_{\mu, m}), \nabla v \big\ra_{G^\ve} \\
+ \frac 1 \mu \big\la  \mathcal B (v), v \big\ra_{V^\prime, V} +  \la \ve f^\ve(t_j,x,u^{j-1}_{\mu, m}) ,  v \ra_{\Gamma^\ve}, 
\end{aligned} 
$$
where $v = \sum_{i=1}^m   \alpha_{mi}^j \psi_i(x)$. 
Assumptions on the nonlinear functions and  monotonicity of  $\mathcal B$  ensure  
\begin{equation}\label{exist_discrit_func}
\begin{aligned} 
J(\alpha) \alpha  \geq & 
 \frac {C_1}{4h}  \| (v + \kappa_D)\chi_{\{v+ \kappa_D >0\}} \|^2_{L^2(G^\ve)} +   \frac {C_2} h  \delta^\beta  \|\nabla v \|^2_{L^2(G^\ve)} 
    - \frac {C_3} {h} \|u^{j-1}_{\mu, m}\chi_{\{v+ \kappa_D >0\}}\|^2_{L^2(G^\ve)}  \\  
    & - \frac {C_4} h \|\nabla u^{j-1}_{\mu, m} \|^2_{L^2(G^\ve)} - C_5\big[\|F^\ve(t_j,x, u^{j-1}_{\mu, m})\|^2_{L^2(G^\ve)}  +  \ve \|f^\ve(t_j,x,u^{j-1}_{\mu, m})\|^2_{L^2(\Gamma^\ve)} \big]  -  \frac{C_6} h \kappa^2_D \\
    & \geq C_7 \big[\| (v + \kappa_D)\chi_{\{v+ \kappa_D >0\}} \|^2_{L^2(G^\ve)} +     \|\nabla v \|^2_{L^2(G^\ve)}\big] - C_8 .
 \end{aligned} 
\end{equation}
 Thus  for sufficiently large $|\alpha|$ we obtain that $J(\alpha) \alpha \geq 0$ and  there exists a zero of $J(\alpha)$ and hence there is a   $u^{j}_{\mu, m} \in \kappa_D + V_m$  satisfying \eqref{penalty_time_discrete_2},  see e.g.\ \cite{Showalter_1996}.  If $b_\delta(v) = \vartheta_b v$, then we have  $\| v + \kappa_D \|^2_{L^2(G^\ve)} $ instead of $\| (v + \kappa_D)\chi_{\{v+ \kappa_D >0\}} \|^2_{L^2(G^\ve)} $. 
 
 Considering $u^j_{\mu, m} -  u^{j-1}_{\mu, m}$ as a test function in \eqref{penalty_time_discrete_2}  and summing  over $j=1, \ldots, l$,  with $1<l \leq N$, yield
 \begin{equation}\label{estim_discrete_111}
 \begin{aligned} 
&  \sum_{j=1}^l  \la  A^\ve(x) k_\delta(u^{j}_{\mu, m}) [ P_{c, \delta}(u^{j}_{\mu, m}) \nabla u^j_{\mu, m} +   \frac 1h  \nabla (u^j_{\mu, m}  - u^{j-1}_{\mu, m}) ], \nabla (u^j_{\mu, m} - u^{j-1}_{\mu, m}) \ra_{G^\ve}  \\ 
 & + \sum_{j=1}^l \frac 1 h \la b_\delta^\prime( u^j_{\mu, m})( u^j_{\mu, m} - u^{j-1}_{\mu, m}) ,  u^j_{\mu, m} -  u^{j-1}_{\mu, m}\ra_{G^\ve} 
  - \sum_{j=1}^l \la F^\ve(t_j,x, u^{j-1}_{\mu, m}), \nabla (u^j_{\mu, m} - u^{j-1}_{\mu, m} )\ra_{G^\ve}  \\
& +   \sum_{j=1}^l \frac 1 \mu \la  \mathcal B (u^{j}_{\mu, m} - \kappa_D), u^j_{\mu, m} -  u^{j-1}_{\mu, m} \ra_{V^\prime, V} 
 = -   \sum_{j=1}^l \ve  \la  f^\ve(t_j,x,u^{j-1}_{\mu, m}) ,  u^j_{\mu, m} - u^{j-1}_{\mu, m} \ra_{\Gamma^\ve}.
\end{aligned} 
 \end{equation} 
 For penalty operator $\mathcal B$   given by 
$
\mathcal B = J(I - P_{\mathcal K^\ve}), 
$
with $P_{\mathcal K^\ve}: V \to \mathcal K^\ve-\kappa_D$ being  the projection operator on $\mathcal K^\ve- \kappa_D$ and $J: V \to V^\prime$  a dual mapping, which can be chosen as 
\begin{equation}\label{dual_project}
\langle J(u), v \rangle_{V^\prime, V} = \int_{G^\ve} \big(u \, v + \nabla u \nabla v\big) dx, 
\end{equation} 
considering that $u_{\mu, m}^0 \in  \mathcal K^\ve$ and using the  property  of the projection operator 
\begin{equation}\label{projection_properties}
\langle J(u-P_{\mathcal K^\ve} u), P_{\mathcal K^\ve} u - v \rangle_{V^\prime, V}  \geq 0 \; \; \text{ for } \;  v \in \mathcal K^\ve - \kappa_D, 
\end{equation}
 we obtain  the following  estimate 
 $$
\begin{aligned} 
 \sum_{j=1}^l \big \langle \mathcal B( u^j_{\mu, m} - \kappa_D), u^j_{\mu, m} -u^{j-1}_{\mu, m} \big  \rangle_{V^\prime, V} 
 & =  \sum_{j=1}^l   \Big[ \Big \langle J(\widetilde  u^j_{\mu, m} -   P_{\mathcal K^\ve} \widetilde u^j_{\mu, m}) , (\widetilde u^j_{\mu, m} -
 P_{\mathcal K^\ve} \widetilde u^j_{\mu, m} )- (\widetilde u^{j-1}_{\mu, m}  -
 P_{\mathcal K^\ve} \widetilde u^{j-1}_{\mu, m} ) \Big \rangle_{V^\prime,V } 
\\ 
& \qquad \qquad + 
\big  \langle J(\widetilde u^j_{\mu, m}  - P_{\mathcal K^\ve} \widetilde u^j_{\mu, m} ),  
 P_{\mathcal K^\ve} \widetilde u^j_{\mu, m} - P_{\mathcal K^\ve} \widetilde u^{j-1}_{\mu, m} \big  \rangle_{V^\prime,V } \Big] 
\\
&\qquad  \geq\frac 12  \int_{G^\ve}\Big[  | (\widetilde u^l_{\mu, m}   - P_{\mathcal K^\ve} \widetilde u^l_{\mu, m})|^2  
+ | \nabla (\widetilde u^l_{\mu, m} - P_{\mathcal K^\ve} \widetilde u^l_{\mu, m})|^2\Big] dx\geq 0, 
\end{aligned} 
$$
where $\widetilde u^j_{\mu, m}= u^j_{\mu, m}-\kappa_D$. 
 Then using in \eqref{estim_discrete_111} the monotonicity of $b$, Lipschitz continuity of $F^\ve$ and $f^\ve$,  regularity of initial data,   and  the uniform boundedness from below of $k_\delta$,  ensures
 $$
 \begin{aligned} 
  \sum_{j=1}^l h \Big\|\frac{\nabla(u^j_{\mu, m} - u^{j-1}_{\mu, m})}{h} \Big\|^2_{L^2(G^\ve)} \leq   C_\sigma \sum_{j=1}^l  h \big(\|\nabla u^j_{\mu, m}\|^2_{L^2(G^\ve)}  + \|u^{j-1}_{\mu, m}\|^2_{L^2(G^\ve)}  + \ve \|u^{j-1}_{\mu, m}\|^2_{L^2(\Gamma^\ve)} \big)
\\ + \sigma_1 \sum_{j=1}^l h \Big \|\frac{\nabla(u^j_{\mu, m} - u^{j-1}_{\mu, m})}{h} \Big\|_{L^2(G^\ve)}^2 
  + \sigma_2 \ve \sum_{j=1}^l h   \Big\|\frac{u^j_{\mu, m} - u^{j-1}_{\mu, m}} {h}\Big\|_{L^2(\Gamma^\ve)}^2 
 \\ 
 \leq C_1\sum_{j=1}^l  h  \sum_{i=1}^j h\Big \|\frac{\nabla(u^i_{\mu, m} - u^{i-1}_{\mu, m})}{h}\Big \|^2_{L^2(G^\ve)}  + \sigma \sum_{j=1}^l h \Big \|\frac{\nabla(u^j_{\mu, m} - u^{j-1}_{\mu, m})}{h} \Big\|_{L^2(G^\ve)}^2 
 + C_2 .
\end{aligned}  
 $$
 In the last estimate we also used the trace and Poincar\'e inequalities. Choosing $\sigma>0$ sufficiently small and applying the discrete Gronwall inequality we obtain 
 \begin{equation}\label{apriori_discrete_11}
 \sum_{j=1}^l h \Big\|\frac{\nabla(u^j_{\mu, m} - u^{j-1}_{\mu, m})}{h} \Big\|^2_{L^2(G^\ve)} \leq  C,  
\end{equation}
 with $1< l \leq N$ and  a constant $C$ independent of $h$, $m$, and $\mu$.   Estimate \eqref{apriori_discrete_11} together with the  Poincar\'e inequality implies 
\begin{equation}\label{estim_disc_deriv} 
\sum_{j=1}^l h \Big\|\frac{u^j_{\mu, m} - u^{j-1}_{\mu, m}}{h} \Big\|^2_{L^2(G^\ve)} \leq C_1 \sum_{j=1}^l h \Big\|\frac{\nabla(u^j_{\mu, m} - u^{j-1}_{\mu, m})}{h} \Big\|^2_{L^2(G^\ve)}  \leq  C.
\end{equation}
Considering now  $u^j_{\mu, m} - \kappa_D $ as a test function in \eqref{penalty_time_discrete_2} yields 
\begin{equation}\label{estim_disct_222}
 \begin{aligned} 
&  \sum_{j=1}^l  \la  A^\ve(x) k_\delta(u^{j}_{\mu, m}) [ P_{c, \delta}(u^{{j}}_{\mu, m}) \nabla u^j_{\mu, m} +   \frac 1h  \nabla (u^j_{\mu, m}  - u^{j-1}_{\mu, m}) ], \nabla u^j_{\mu, m}  \ra_{G^\ve}  \\ 
& +  \sum_{j=1}^l \frac 1 h \la b_\delta^\prime( u^j_{\mu, m})(u^j_{\mu, m} -  u^{j-1}_{\mu, m}),  u^j_{\mu, m} -  \kappa_D\ra_{G^\ve} 
   - \sum_{j=1}^l \la F^\ve(t,x, u^{j-1}_{\mu, m}), \nabla u^j_{\mu, m}  \ra_{G^\ve} 
\\
&+   \frac 1 \mu \sum_{j=1}^l \la  \mathcal B (u^{j}_{\mu, m} - \kappa_D), u^j_{\mu, m} -  \kappa_D \ra_{V^\prime, V} 
= -   \sum_{j=1}^l \la \ve f^\ve(t,x,u^{j-1}_{\mu, m}) ,  u^j_{\mu, m} -\kappa_D \ra_{\Gamma^\ve}.
\end{aligned} 
 \end{equation}
Then assumptions on $A$,  $k$, $P_c$, $b$, $F^\ve$ and $f^\ve$, together with the  trace and Poincar\'e inequalities, monotonicity of $\mathcal B$,  and estimates \eqref{apriori_discrete_11} and \eqref{estim_disc_deriv},   ensure 
\begin{equation}\label{apriori_33}
\sum_{j=1}^l h \big[\| \nabla u^j_{\mu, m}\|^2_{L^2(G^\ve)} +  \| u^j_{\mu, m}\|^2_{L^2(G^\ve)} \big] + 
\frac 1 \mu \sum_{j=1}^l h  \la  \mathcal B (u^{j}_{\mu, m} - \kappa_D), u^j_{\mu, m} -  \kappa_D \ra_{V^\prime, V}  \leq C , 
\end{equation}
with a constant $C$ independent of $\mu$, $m$,  and $h$. 
The second term in \eqref{estim_disct_222} is estimated,  using the assumptions on  $b$ and the continuous embedding  $H^1(G^\ve)\subset L^6(G^\ve)$   for $n \leq 3$, in the following way 
$$
\begin{aligned} 
& \sum_{j=1}^l h \big |\la b_\delta^\prime( u^j_{\mu, m}) \frac 1 h(u^j_{\mu, m} -  u^{j-1}_{\mu, m}),  u^j_{\mu, m} -  \kappa_D\ra_{G^\ve} \big |
 \leq  C_1 \sum_{j=1}^l   h  \Big\|\frac{u^j_{\mu, m} - u^{j-1}_{\mu, m}}{h} \Big\|^2_{L^2(G^\ve)}
 \\ &  + 
C_2 \sum_{j=1}^l h  \big(\| u^j_{\mu, m}\|^6_{L^6(G^\ve)}  + 1 \big) \leq C_3  \sum_{j=1}^l  h  \Big\|\frac{u^j_{\mu, m} - u^{j-1}_{\mu, m}}{h} \Big\|^2_{L^2(G^\ve)} 
+ C_4  \Big[ \sum_{j=1}^l  h  \Big\|\frac{\nabla(u^j_{\mu, m} - u^{j-1}_{\mu, m})}{h} \Big\|^2_{L^2(G^\ve)} \Big]^3 + C_5 .  
\end{aligned} 
$$

To show that a subsequence of approximate solutions  $\{u_{\mu, m}^j\}$ converges to a solution of problem \eqref{penalty} we define piecewise linear and piecewise constant interpolations with respect to  the time variable 
$$
\begin{aligned} 
&u^N_{\mu, m}(t,x) := u_{\mu, m}^{j-1}(x) + (t - t_{j-1}) \frac{ u_{\mu, m}^j(x)  -  u_{\mu, m}^{j-1}(x)} h \quad  && \text{ for }  t \in (t_{j-1}, t_j], \\
&\bar u^N_{\mu, m} (t,x) : = u_{\mu, m}^j(x) \quad &&\text{ for }  t \in (t_{j-1}, t_j].
\end{aligned}
$$
Then a priori estimates in \eqref{apriori_discrete_11}, \eqref{estim_disc_deriv},  and \eqref{apriori_33} and the boundedness of  the penalty operator  $\mathcal B$ ensure 
\begin{equation}\label{a_priori_22} 
\begin{aligned} 
\|\bar u^N_{\mu, m}\|_{L^2(G^\ve_T)} +  \|\nabla \bar u^N_{\mu, m}\|_{L^2(G^\ve_T)} 
+ \|\partial_t u^N_{\mu, m}\|_{L^2(G^\ve_T)}  + \|\partial_t \nabla  u^N_{\mu, m}\|_{L^2(G^\ve_T)}  \leq C,  \\
  \int_0^T \|\mathcal B( \bar u^N_{\mu, m}  - \kappa_D)\|^2_{V^\prime}  dt \leq C, 
\end{aligned} 
\end{equation} 
with a constant $C$ independent of $N$, $m$, and $\mu$. Integrating problem \eqref{penalty_time_discrete_2} over $(0,T)$ yields
 \begin{equation}\label{discrete_2}
\begin{aligned} 
&  \la  b_\delta^\prime(\bar u^{N}_{\mu, m})  \partial_t  u^N_{\mu, m}, \zeta \ra_{G^\ve_T} 
  + \la  A^\ve(x) k_\delta(\bar u^{N}_{\mu, m}) [ P_{c, \delta}(\bar u^{N}_{\mu, m}) \nabla \bar u^N_{\mu, m} +  \partial_t  \nabla u^N_{\mu, m} ], \nabla \zeta \ra_{G^\ve_T}  \\ 
&  -  \la F^\ve(t,x, \bar u^{N,h}_{\mu, m}), \nabla \zeta \ra_{G^\ve_T} 
+ \frac 1 \mu \int_0^T \la   \mathcal B (\bar u^N_{\mu, m} - \kappa_D), \zeta \ra_{V^\prime, V} dt  = -   \la \ve f^\ve(t,x, \bar u^{N,h}_{\mu, m}) ,  \zeta \ra_{\Gamma^\ve_T}, 
\end{aligned} 
\end{equation}
 for $\zeta \in L^2(0, T; V_m)$, where $\bar u^{N,h}_{\mu, m}(t,x) = \bar u^{N}_{\mu, m}(t-h, x)$ for $t \in [h, T]$ and $\bar u^{N,h}_{\mu, m}(t,x) = u^0_{\mu,m}(x)$ for $t\in[0,h]$ and $x\in G^\ve$. 

A priori estimates \eqref{a_priori_22} imply that there exist $u_\mu \in H^1(0, T; H^1(G^\ve))$  and $\Lambda \in L^2(0,T; V^\prime)$ such that, up to a subsequence, 
\begin{equation}\label{convergence_22} 
\begin{aligned}
 \bar u^N_{\mu, m} &\rightharpoonup \,  u_\mu && \text{weakly in } L^2(0, T; H^1(G^\ve)),  \text{ strongly in } L^2(0,T; H^\sigma(G^\ve)),   \\
 u^N_{\mu, m} & \rightharpoonup \,  u_\mu && \text{weakly}-\ast \text{ in } L^\infty(0, T; H^1(G^\ve)),   \text{ strongly in } L^2(0,T;  H^\sigma(G^\ve)),   \\
 \partial_t u^N_{\mu, m}  &\rightharpoonup \,   \partial_t  u_\mu &&  \text{weakly  in } L^2(0, T; H^1(G^\ve)),  \\
  \mathcal B (\bar u^N_{\mu, m} - \kappa_D)  & \rightharpoonup \,  \Lambda &&  \text{weakly in }  L^2(0,T; V^\prime), 
\end{aligned} 
\end{equation} 
as $N, m \to \infty$,  where $1/2 < \sigma < 1$, 
and 
$$
\| \bar u^N_{\mu, m} -  \bar u^{N, h}_{\mu, m} \|_{L^2(0, T: H^1(G^\ve))} \leq \frac {C}{\sqrt{N} }. 
$$
Using a priori estimates \eqref{a_priori_22} we also obtain 
\begin{equation}\label{estim_bt}
\begin{aligned} 
\|b_\delta^\prime(\bar u^{N}_{\mu, m})  \partial_t  u^N_{\mu, m}\|_{L^2(G^\ve_T)}^2 & \leq C_1 \int_0^T \big(\|\bar u^{N}_{\mu, m}\|^4_{L^6(G^\ve)} +1+ \delta^{4(\alpha -1)} \big) \|\partial_t  u^N_{\mu, m}\|_{L^6(G^\ve)}^2 dt  \\
& \leq C_2\big(\|\nabla \bar u^{N}_{\mu, m}\|^4_{L^\infty(0,T; L^2(G^\ve))} +C_\delta \big)\|\partial_t \nabla u^N_{\mu, m}\|_{L^2(G^\ve_T)}^2 \leq C.
 \end{aligned} 
\end{equation}
Taking in \eqref{discrete_2} the  limit as $N, m \to \infty$ and using convergence results in \eqref{convergence_22}, together with the continuity of nonlinear functions,   we obtain  
 \begin{equation}\label{discrete_22}
\begin{aligned} 
& \la  b_\delta^\prime(u_\mu)  \partial_t  u_{\mu}, \zeta \ra_{G^\ve_T} 
  + \la  A^\ve(x) k_\delta(u_{\mu}) [ P_{c, \delta}(u_{\mu}) \nabla u_{\mu} +  \partial_t  \nabla u_{\mu} ], \nabla \zeta \ra_{G^\ve_T}  \\ 
  & \qquad -  \la F^\ve(t,x,  u_{\mu}), \nabla \zeta \ra_{G^\ve_T} 
+ \frac 1 \mu \int_0^T \la  \Lambda, \zeta \ra_{V^\prime, V} dt = -    \ve \la f^\ve(t,x, u_{\mu}) ,  \zeta \ra_{\Gamma^\ve_T}, 
\end{aligned} 
\end{equation}
 for $\zeta \in L^2(0,T;  V)$. 
To show strong convergence of $\bar u^N_{\mu, m}$ to $u_\mu$ in $L^2(0, T; H^1(G^\ve))$ we consider $\int_{u_\mu}^{\bar u^N_{\mu, m}} \frac 1 {k_\delta (\xi)} d\xi $ as a test function in  \eqref{discrete_2} and  obtain 
$$
\begin{aligned} 
&\big \la  A^\ve(x)  \nabla (\bar u_{\mu, m}^N - u_\mu )(s), \nabla (\bar u_{\mu, m}^N - u_\mu)(s)  \big \ra_{G^\ve}
  + \big \la A^\ve(x) P_{c, \delta}(\bar u_{\mu, m}^{N}) \nabla (\bar u_{\mu, m}^N  -u_\mu), \nabla (\bar u_{\mu, m}^N - u_\mu)  \big\ra_{G^\ve_s} 
\\ 
&  + \frac 1 \mu  \int_0^s \Big  \la  \mathcal B(\bar u_{\mu, m}^N - \kappa_D) - \mathcal B(u_\mu - \kappa_D), \int_{u_\mu}^{\bar u^N_{\mu, m}} \frac 1 {k_\delta (\xi)} d\xi  \Big \ra_{V^\prime, V} dt 
\\ 
& =  \big \la  A^\ve(x)  \nabla (u^0_{\mu, m} -  u_0 ), \nabla (u^0_{\mu, m} - u_0) \big \ra_{G^\ve}
-\big \la  A^\ve(x) \partial_t \nabla u_\mu , \nabla (\bar u_{\mu, m}^N -  u_\mu) \big \ra_{G^\ve_s}
\\ 
& -\Big\la  A^\ve(x)  [ P_{c, \delta}(\bar u_{\mu, m}^{N}) \nabla \bar u_{\mu, m}^N +  \partial_t  \nabla u_{\mu, m}^N ],  \Big[1 -\frac{ k_\delta(\bar u_{\mu, m}^N)}{ k_\delta(u_\mu)} \Big] \nabla u_\mu  \Big\ra_{G^\ve_s}  
    - \Big \la  \partial_t b_\delta(\bar u_{\mu, m}^N) , \int_{u_\mu}^{\bar u^N_{\mu, m}} \frac 1 {k_\delta (\xi)} d\xi  \Big\ra_{G^\ve_s} 
  \\ 
  &  -  \la A^\ve(x)  P_{c, \delta}(\bar u_{\mu, m}^{N}) \nabla u_\mu, \nabla \bar u_{\mu, m}^N - \nabla u_\mu  \ra_{G^\ve_s} 
  +   \Big\la F^\ve(t,x,  \bar u_{\mu, m}^{N,h}),   \frac 1 {k_\delta(\bar u_{\mu, m}^N)} \nabla \bar u^N_{\mu, m} - \frac 1 {k_\delta(u_\mu)} \nabla u_\mu \Big\ra_{G^\ve_s} 
\\ 
& - \frac 1 \mu \int_0^s \Big  \la  \mathcal B(u_\mu - \kappa_D), \int_{u_\mu}^{\bar u^N_{\mu, m}} \frac 1 {k_\delta (\xi)} d\xi   \Big \ra_{V^\prime, V} dt
-  \ve  \Big  \la f^\ve(t,x, \bar u_{\mu, m}^{N,h}) ,  \int_{u_\mu}^{\bar u^N_{\mu, m}} \frac 1 {k_\delta (\xi)} d\xi  \Big \ra_{\Gamma^\ve_s}, 
 \end{aligned} 
$$
for $s \in (0, T]$. Then using the following estimate for the penalty operator $\mathcal B$
\begin{equation}\label{positive_B}
 \int_0^s \Big\la  \mathcal B(\bar u_{\mu, m}^N - \kappa_D) - \mathcal B(u_\mu - \kappa_D), \int_{u_\mu}^{\bar u^N_{\mu, m}} \frac 1 {k_\delta (\xi)} d\xi  \Big\ra_{V^\prime, V} dt  \geq 0,  
 \end{equation}
for $s \in (0, T]$,   shown below,   the  strong convergence of $u_{\mu, m}^N$ in $L^2(G^\ve_T)$ and weak convergence  in $H^1(0,T; H^1(G^\ve))$ as $m,N \to \infty$,  together with the continuity of nonlinear functions and assumptions on $A^\ve$ and $P_{c}$, imply
 $$
 \sup_{(0,T)}\|\nabla (\bar u_{\mu, m}^N - u_\mu)\|_{L^2(G^\ve)}  \to 0 \quad \text{ as } \; m, N \to \infty. 
 $$
 To show  \eqref{positive_B} we consider  
 $$
 \begin{aligned} 
& \frac 1 {k(\delta)} \int_0^T \Big\la  \mathcal B(\bar u_{\mu, m}^N - \kappa_D) - \mathcal B(u_\mu - \kappa_D), k(\delta) \int_{u_\mu}^{\bar u^N_{\mu, m}} \frac 1 {k_\delta (\xi)} d\xi \Big \ra_{V^\prime, V} dt 
\\
&  = \frac 1 {k(\delta)}   \int_0^T \Big\la  \mathcal B(\bar u_{\mu, m}^N - \kappa_D) - \mathcal B(u _\mu- \kappa_D), (\bar u_{\mu, m}^N  - P_{\mathcal K^\ve}(\bar u_{\mu, m}^N - \kappa_D) )- (u_\mu -P_{\mathcal K^\ve}(u_\mu - \kappa_D)) \Big \ra_{V^\prime, V} dt 
 \\
& + \frac  1{k( \delta)}  \int_0^T \Big\la  \mathcal B(\bar u_{\mu, m}^N - \kappa_D) - \mathcal B(u_\mu - \kappa_D),  P_{\mathcal K^\ve}(\bar u_{\mu, m}^N - \kappa_D) -P_{\mathcal K^\ve}(u_\mu - \kappa_D) \\
& \hspace{7 cm} - (\bar u_{\mu, m}^N - u_\mu) + k(\delta) \int_{u_\mu}^{\bar u^N_{\mu, m}} \frac {d\xi} {k_\delta (\xi)}   \Big \ra_{V^\prime, V} dt, 
 \end{aligned} 
 $$
 where $k(\delta) >0$. 
 The monotonicity of $\mathcal B $ ensures 
$$
\int_0^T \Big \la  \mathcal B(\bar u_{\mu, m}^N - \kappa_D) - \mathcal B(u _\mu- \kappa_D), (\bar u_{\mu, m}^N  - P_{\mathcal K^\ve}(\bar u_{\mu, m}^N - \kappa_D) )- (u_\mu -P_{\mathcal K^\ve}(u_\mu - \kappa_D)) \Big\ra_{V^\prime, V} dt \geq 0. 
$$
For the second  term   due to the properties of the projection operator  we have
$$
 \begin{aligned} 
 \int_0^T \Big \la  \mathcal B(\bar u_{\mu, m}^N - \kappa_D),  P_{\mathcal K^\ve}(\bar u_{\mu, m}^N - \kappa_D) - \Big[P_{\mathcal K^\ve}(u_\mu - \kappa_D)  + (\bar u_{\mu, m}^N - u_\mu)  - k(\delta) \int_{u_\mu}^{\bar u^N_{\mu, m}} \frac 1 {k_\delta (\xi)} d\xi  \Big] \Big \ra_{V^\prime, V} dt \geq 0 
 \end{aligned} 
 $$
and 
$$
 \begin{aligned} 
 \int_0^T \Big \la   \mathcal B(u_\mu - \kappa_D),  P_{\mathcal K^\ve}(u_\mu - \kappa_D)  - \Big[ P_{\mathcal K^\ve}(\bar u_{\mu, m}^N - \kappa_D) +  k(\delta) \int_{u_\mu}^{\bar u^N_{\mu, m}} \frac 1 {k_\delta (\xi)} d\xi  - (\bar u_{\mu, m}^N - u_\mu) \Big]  \Big \ra_{V^\prime, V} dt \geq 0,  
 \end{aligned} 
 $$
 if   $
P_{\mathcal K^\ve}(u_\mu - \kappa_D) +(\bar u_{\mu, m}^N - u_\mu)  - k(\delta) \int_{u_\mu}^{\bar u^N_{\mu, m}}  k_\delta (\xi)^{-1} d\xi    \in \mathcal K^\ve - \kappa_D
$ and  $P_{\mathcal K^\ve}(\bar u_{\mu, m}^N - \kappa_D) + k(\delta) \int_{u_\mu}^{\bar u^N_{\mu, m}}  k_\delta (\xi)^{-1} d\xi  - (\bar u_{\mu, m}^N - u_\mu)   \in \mathcal K^\ve - \kappa_D$, respectively. 
 Notice that if $u_\mu \leq 0$ and $\bar u_{\mu,m}^N \leq 0$, then $k(\delta) \int_{u_\mu}^{\bar u^N_{\mu, m}}  {k_\delta (\xi)^{-1}} d\xi  = \bar u^N_{\mu, m}- u_\mu$.
  If 
 $\bar u_{\mu,m}^N>u_\mu$ then  $(\bar u_{\mu, m}^N - u_\mu)  - k(\delta) \int_{u_\mu}^{\bar u^N_{\mu, m}}  {k_\delta (\xi)^{-1}} d\xi \geq 0 $
 and if  $\bar u_{\mu,m}^N<u_\mu$ and $u_\mu>0$,  then  for $\bar u^N_{\mu, m} \leq 0$ we have  
 $k(\delta) \int_{0}^{\bar u^N_{\mu, m}}  k_\delta (\xi)^{-1} d\xi  = \bar u^N_{\mu, m} $ and hence 
 $\bar u_{\mu, m}^N   -k(\delta) \int_{u_\mu}^{\bar u^N_{\mu, m}}  k_\delta (\xi)^{-1} d\xi  \geq 0$ and $P_{\mathcal K^\ve}(u_\mu - \kappa_D)  = u_\mu - \kappa_D$.  Thus  combining those considerations yields   $
P_{\mathcal K^\ve}(u_\mu - \kappa_D) +(\bar u_{\mu, m}^N - u_\mu)  - k(\delta) \int_{u_\mu}^{\bar u^N_{\mu, m}}  k_\delta (\xi)^{-1} d\xi    \in \mathcal K^\ve - \kappa_D$.  
 For the second term,  if $u_\mu>u_{\mu,m}^N$ then  $ k(\delta) \int_{u_\mu}^{u^N_{\mu, m}} k_\delta (\xi)^{-1} d\xi  - (u_{\mu, m}^N - u_\mu)>0$ and if 
 $u_\mu < \bar u_{\mu, m}^N$ and $\bar u_{\mu, m}^N >0$, then since for $u_\mu <0$ we have $k(\delta) \int_{u_\mu}^0  k_\delta(\xi)^{-1} d\xi = - u_{\mu}$ and hence   $k(\delta) \int_{u_\mu}^{\bar u^N_{\mu, m}} k_\delta(\xi)^{-1} d\xi + u_{\mu} \geq 0$, we obtain 
 $P_{\mathcal K^\ve} (\bar u_{\mu, m}^N - \kappa_D) +  k(\delta) \int_{u_\mu}^{\bar u^N_{\mu, m}} k_\delta (\xi)^{-1} d\xi  - (\bar u_{\mu, m}^N - u_\mu )   \in \mathcal K^\ve - \kappa_D$.   Notice that for $\bar u_{\mu, m}^N>0$ we have  $P_{\mathcal K^\ve} (\bar u_{\mu, m}^N - \kappa_D) = \bar u_{\mu, m}^N - \kappa_D$. Thus  inequality  \eqref{positive_B} follows. 

The strong convergence of $\bar u_{\mu, m}^N$ in $L^2(0,T; H^1(G^\ve))$  implies  $\mathcal B(\overline u_{\mu, m}^N - \kappa_D) \rightharpoonup \mathcal B(u_\mu - \kappa_D)$ 
in $L^2(0,T; V^\prime)$  as $m, N \to \infty$, and hence $\Lambda = \mathcal B(u_\mu - \kappa_D)$.  Therefore we obtain  that $u_\mu$ is a weak solution of problem \eqref{penalty}.

To prove  the existence of a solution of variational inequality \eqref{main2_regul}  we need to  take  in \eqref{penalty}  the limit as $\mu \to 0$.  Notice that  a priori estimates \eqref{a_priori_22} and \eqref{estim_bt}  are uniform in $\mu$. Hence taking the limit as $N,m\to \infty$ and using lower semicontinuity of a norm we obtain the corresponding  estimates for $u_\mu$ in $H^1(0,T; H^1(G^\ve))$ and that there exists $u \in H^1(0,T; H^1(G^\ve))$ such that, up to a subsequence, $u_\mu \rightharpoonup u$ in $H^1(0,T; H^1(G^\ve))$ as $\mu \to 0$. Assumptions on $b$, $k$,  and $P_c$ and strong convergence of $u_\mu \to u$ in $L^{r_1}(G_T^\ve)$, for $1<  r_1 < 6$, ensure strong convergence  $b_\delta(u_\mu) \to b_\delta(u)$ in $L^{r_2}(G_T^\ve)$, for $1<r_2<2$,  $k_\delta(u_\mu) \to k_\delta(u)$,  $k_\delta(u_\mu) P_{c, \delta}(u_\mu) \to k_\delta(u) P_{c, \delta}(u)$ in $L^q(G_T^\ve)$, for $1< q < \infty$,  as $\mu \to 0$,  and $\partial_t b_\delta(u) \in L^2(G_T^\ve)$.  From  equation \eqref{discrete_22} follows  
$$
 \begin{aligned} 
\int_0^T \la \mathcal B( u_\mu - \kappa_D) , v \ra_{V^\prime, V} dt = \mu \int_0^T \Big[\big\la F^\ve(t,x, u_\mu) - A^\ve(x) \, k_\delta(u_\mu)\big[P_{c,\delta} (u_\mu) \nabla u_\mu + \partial_t \nabla u_\mu\big] , \nabla v \big\ra_{G^\ve} \\  - \ve \la f^\ve(t,x, u_\mu), v \ra_{\Gamma^{\ve}}  - \la \partial_t b_{\delta} (u_\mu) , v \ra_{G^\ve} \Big]dt 
 \end{aligned} 
 $$
for all $v\in L^2(0,T; V)$. Then  boundedness of  $u_\mu$ in $H^1(0, T; H^1(G^\ve))$ yields
\begin{equation}\label{converg_B}
\mathcal B( u_\mu - \kappa_D)  \rightharpoonup  0  \quad \text{weakly  in } \;  L^2(0,T; V^\prime)\; \; \;  \text{ as }  \;  \mu \to 0.  
\end{equation}
The monotonicity  of $\mathcal B$ ensures 
$$
\int_0^T\la \mathcal B(v), u_\mu - \kappa_D - v\ra_{V^\prime, V}  dt \leq \int_0^T\la \mathcal B(u_\mu - \kappa_D), u_\mu  - \kappa_D - v\ra_{V^\prime, V}  dt
$$
for   $v\in L^2(0,T; V)$.  Considering $\mu \to 0$ and using weak convergence of $u_\mu \rightharpoonup u$ in $L^2(0,T; H^1(G^\ve))$ as $\mu \to 0$,   convergence of $\mathcal B(u_\mu - \kappa_D)$, see \eqref{converg_B},  and  the fact that 
$$
\int_0^T\la \mathcal B(u_\mu - \kappa_D), u_\mu- \kappa_D\ra_{V^\prime, V}  dt \leq C \mu 
$$
imply 
$$
\int_0^T\la \mathcal B(v), u - \kappa_D - v\ra_{V^\prime, V}  dt \leq 0 . 
$$
 Taking $v= u - \kappa_D- \lambda w$ for $\lambda >0$ and  $w\in L^2(0,T; V)$, passing to the limit as $\lambda \to 0$,  and using hemicontinuity of $\mathcal B$ we obtain 
$$
\int_0^T\la \mathcal B(u - \kappa_D), w\ra_{V^\prime, V}  dt \leq 0 
$$
for all $w\in L^2(0,T; V)$ and hence $\mathcal B(u - \kappa_D) = 0 $ and  $u(t) \in \mathcal K^\ve$ for a.a.\ $t\in (0,T)$. 

To  show that $u$ is a solution of variational inequality \eqref{main2_regul}  we consider $\zeta = v - u- k(\delta) \int_{u}^{u_\mu} \frac 1{k_\delta(\xi)} d\xi$ as a test function in \eqref{discrete_22}  and obtain 
\begin{equation}\label{discrete_33}
\begin{aligned} 
 \Big\la  \partial_t b_\delta(u_\mu), v  - u- k(\delta) \int_{u}^{u_\mu} \frac {d\xi}{k_\delta(\xi)} \Big \ra_{G^\ve_T}  -  \Big \la F^\ve(t,x,  u_{\mu}), \nabla(v -  u)  - \frac {k(\delta)} {k_\delta(u_\mu)} \nabla u_\mu +  \frac{k(\delta)} {k_\delta(u)} \nabla u \Big \ra_{G^\ve_T}
 \\
  +  \Big \la  A^\ve(x) k_\delta(u_{\mu}) [ P_{c, \delta}(u_{\mu}) \nabla u_{\mu} +  \partial_t  \nabla u_{\mu} ], \nabla (v - u)  - \frac {k(\delta)} {k_\delta(u_\mu)} \nabla u_\mu +  \frac {k(\delta)} {k_\delta(u)} \nabla u \Big  \ra_{G^\ve_T}  \\    
+    \ve\Big \la f^\ve(t,x, u_{\mu}) ,  v - u  - k(\delta) \int_{u}^{u_\mu} \frac {d\xi}{k_\delta(\xi)}\Big\ra_{\Gamma^\ve_T} = 
\frac 1 \mu \int_0^T \Big \la  \mathcal B(u_\mu - \kappa_D),  u + k(\delta) \int_{u}^{u_\mu} \frac {d\xi} {k_\delta(\xi)} - v \Big\ra_{V^\prime, V} dt 
\end{aligned} 
\end{equation}
for $v\in L^2(0,T; \mathcal K^\ve)$. In order to pass to the limit as $\mu \to 0$ we need to show that 
$$
\int_0^T \Big \la   \mathcal B(u_\mu - \kappa_D), u+  k(\delta) \int_{u}^{u_\mu} \frac {d\xi} {k_\delta(\xi)}    - v  \Big\ra_{V^\prime, V} dt \geq 0 
$$
for $v \in  L^2(0,T; \mathcal K^\ve)$.   Since $\mathcal B(v - \kappa_D) = 0$ we can rewrite the left had side in the last inequality  as 
\begin{equation}\label{monot_B_22}
\begin{aligned} 
& \int_0^T \Big \la   \mathcal B(u_\mu - \kappa_D) - \mathcal B(v - \kappa_D), (u_\mu - \kappa_D)  - (v - \kappa_D) \Big\ra_{V^\prime, V} dt \\
& + \int_0^T \Big \la   \mathcal B(u_\mu - \kappa_D),  P_{\mathcal K^\ve}(u_\mu - \kappa_D) - \Big [P_{\mathcal K^\ve}(u_\mu - \kappa_D)  +  (u_\mu -  u ) -  k(\delta) \int^{u_\mu}_u \frac {d\xi}{k_\delta(\xi)}  \Big] \Big\ra_{V^\prime, V} dt .
\end{aligned} 
\end{equation}
The first term in \eqref{monot_B_22} is nonnegative due to the monotonicity of $\mathcal B$, whereas  the second term is nonnegative  if $q_\mu =  P_{\mathcal K^\ve}(u_\mu - \kappa_D) +(u_\mu - u) -  k(\delta) \int^{u_\mu}_u  k_\delta(\xi)^{-1} d\xi     \in \mathcal K^\ve - \kappa_D$. First notice that $u\in \mathcal K^\ve$ and hence $u\geq 0$ on $\Gamma^\ve$. 
If $u_\mu \geq u$  then  due to assumptions on $k$ we have  $( u_\mu - u ) -  k(\delta) \int^{u_\mu}_u  k_\delta(\xi)^{-1} d\xi   \geq 0$ and  hence $q_\mu \in \mathcal K^\ve - \kappa_D$.
If $u=u_\mu = 0$  on $\Gamma^\ve$  or if  $u =0$ and $u_\mu \leq 0$ on $\Gamma^\ve$ we obtain
$ u_\mu - k(\delta) \int_0^{u_\mu} k_\delta(\xi)^{-1} d\xi =0$ and $q_\mu = P_{\mathcal K^\ve}(u_\mu - \kappa_D)  \in \mathcal K^\ve - \kappa_D$. If $u>0$ and $u_\mu < u$ on $\Gamma^\ve$, then, since $u_\mu \to u$ in $L^2(\Gamma^\ve_T)$ as $\mu \to 0$,  there exists such $\mu>0$ that $0< u_\mu  \leq u$  and $|u - u_\mu| \leq u_\mu$ a.e.\ on  $(0,T)\times\Gamma^\ve$,  and thus
$q_\mu \geq u_\mu - \kappa_D - (u- u_\mu) \geq - \kappa_D$ and  $q_\mu \in \mathcal K^\ve - \kappa_D$.   

Considering the limit as $\mu \to 0$ in \eqref{discrete_33} and integration by parts in   
$ \la  A^\ve(x)  \partial_t  \nabla u_{\mu} , \nabla u_\mu   \ra_{G^\ve_T} $, combined with   strong convergence of $u_\mu$ in $L^p((0,T)\times G^\ve)$ for any $1<p<6$, positivity of functions $k_\delta$ and $P_{c,\delta}$,  continuity of nonlinear functions and lower semicontinuity of a norm,   yield 
\begin{equation}\label{discrete_44}
\begin{aligned} 
 \big\la  \partial_t b_\delta(u), v  - u \big \ra_{G^\ve_T} 
  +   \big\la  A^\ve(x) k_\delta(u) [ P_{c, \delta}(u) \nabla u +  \partial_t  \nabla u ], \nabla (v  -u)  \big\ra_{G^\ve_T}  
  -  \big \la F^\ve(t,x,  u), \nabla (v -   u)  \big \ra_{G^\ve_T}   \\ +    \ve\big \la f^\ve(t,x, u) ,  v - u \big\ra_{\Gamma^\ve_T} \geq 0. 
\end{aligned} 
\end{equation}
Thus we obtain  that  $u_\delta^\ve= u$ is a solution of variational inequality  \eqref{main2_regul}. 

To show the uniqueness of a solution of variational inequality \eqref{main2_regul}  we assume that there are two solutions $u^\ve_{\delta, 1}$ and $u^\ve_{\delta, 2}$  and  consider $v = u_{\delta, 2}^\ve$ and $v=u_{\delta, 1}^\ve$ as test functions in variational inequalities for $u^\ve_{\delta, 1}$ and $u^\ve_{\delta, 2}$, respectively,
\begin{equation}\label{Unique_1}
\begin{aligned}
\big \langle A^\ve(x) \big(k_\delta(u^\ve_{\delta, 1}) [ P_{c,\delta}(u^\ve_{\delta, 1}) \nabla u_{\delta, 1}^\ve +  \partial_t \nabla u_{\delta, 1}^\ve] -  k_\delta(u^\ve_{\delta, 2}) [ P_{c, \delta}(u^\ve_{2, \delta}) \nabla u_{2, \delta}^\ve +  \partial_t \nabla u_{2, \delta}^\ve]\big), \nabla (u_{\delta, 1}^\ve-u_{\delta, 2}^\ve)\big \rangle_{G^\ve_\tau}  \\  
 +  \big \langle \partial_t (b_\delta(u^\ve_{\delta, 1})- b_\delta(u^\ve_{\delta, 2})),   u_{\delta, 1}^\ve - u_{\delta, 2}^\ve \big \rangle_{G^\ve_\tau}  
 - \big \langle   F^\ve(t,x, u^\ve_{\delta, 1}) -  F^\ve(t,x, u^\ve_{\delta, 2}),  \nabla (u^\ve_{\delta, 1} - u^\ve_{\delta, 2}) \big \rangle_{G^\ve_\tau}  \\ + 
\ve \big\langle  f^\ve(t,x,u^\ve_{\delta, 1}) -  f^\ve(t,x,u^\ve_{\delta, 2}),  u^\ve_{\delta, 1} - u^\ve_{\delta, 2}  \big\rangle_{\Gamma^\ve_\tau}  \leq 0, 
 \end{aligned}
\end{equation}
for $\tau \in (0,T]$. Rearranging  terms in  \eqref{Unique_1}   implies  
\begin{eqnarray}\label{Unique_2}
&& \frac 12 \int_{G^\ve_\tau}  A^\ve(x) \partial_t  \big(k_\delta(u^\ve_{\delta, 1})|\nabla u^\ve_{\delta, 1}- \nabla u^\ve_{\delta,2}|^2\big) dx dt  
- \frac 12  \int_{G^\ve_\tau}  A^\ve(x) \partial_t k_\delta(u^\ve_{\delta, 1})|\nabla u^\ve_{\delta, 1}- \nabla u^\ve_{\delta, 2}|^2 dx dt  \nonumber \\
&& + \big \langle \partial_t (b_\delta(u^\ve_{\delta, 1})- b_\delta(u^\ve_{\delta, 2})),  u^\ve_{\delta, 1} - u^\ve_{\delta,  2}  \big \rangle_{G^\ve_\tau}  + \left \langle A^\ve(x) (k_\delta(u^\ve_{\delta, 1}) - k_\delta(u_{\delta, 2}^\ve))\partial_t \nabla u^\ve_{\delta, 2}, \nabla (u^\ve_{\delta, 1}-  u^\ve_{\delta, 2})\right \rangle_{G^\ve_\tau}
\nonumber \\
&& + \left \langle A^\ve(x) k_\delta(u^\ve_{\delta, 1}) P_{c, \delta} (u^\ve_{\delta, 1})\nabla (u_{\delta, 1}^\ve  -  u_{\delta, 2}^\ve),  \nabla (u^\ve_{\delta, 1}-u^\ve_{\delta, 2})\right \rangle_{G^\ve_\tau} 
\\
&& +  \big \langle A^\ve(x) (k_\delta(u^\ve_{\delta, 1}) P_{c, \delta}(u^\ve_{\delta, 1}) - k(u^\ve_{\delta, 2}) P_{c, \delta}(u^\ve_{\delta, 2})) \nabla u_{\delta, 2}^\ve,  \nabla (u^\ve_{\delta, 1}-u^\ve_{\delta, 2})\big \rangle_{G^\ve_\tau} 
\nonumber \\
 && -\big \langle   F^\ve(t,x, u^\ve_{\delta, 1}) -  F^\ve(t,x, u^\ve_{\delta, 2}),  \nabla (u^\ve_{\delta, 1} - u^\ve_{\delta, 2})\big \rangle_{G^\ve_\tau}   + 
\ve \big\langle  f^\ve(t,x,u^\ve_{\delta, 1}) -  f^\ve(t,x,u^\ve_{\delta, 2}),  u^\ve_{\delta, 1} - u_{\delta, 2}^\ve\big \rangle_{\Gamma^\ve_\tau}  \leq 0,  \nonumber
\end{eqnarray}
for $\tau \in (0,T]$.  Using regularity assumptions  on $\partial_t u^\ve_{\delta, 1}$, the Lipschitz continuity of  $k$ and boundedness of  $A^\ve$, the second term in \eqref{Unique_2}  can be estimated as
$$
\Big| \int_{G^\ve_\tau}  A^\ve(x) \partial_t k_\delta(u^\ve_{\delta, 1})|\nabla u^\ve_{\delta, 1}- \nabla u^\ve_{\delta, 2}|^2 dx dt \Big| \leq  C \sup_{(0,\tau)}  \|\nabla u^\ve_{\delta, 1}- \nabla u^\ve_{\delta, 2}\|^2_{L^2(G^\ve)} 
 \tau^{\frac 1 2} \|\partial_t u^\ve_{\delta, 1}\|_{L^2(0, T; L^\infty(G^\ve))}.
 $$
 The third term in \eqref{Unique_2}  is estimated  as
 \begin{equation}\label{estim_b_22}
\begin{aligned}
 \langle \partial_t (b_\delta(u^\ve_{\delta, 1})- b_\delta(u^\ve_{\delta, 2})),  u^\ve_{\delta, 1} - u^\ve_{\delta, 2} \rangle_{G^\ve_\tau}    = 
  \big\langle  b^\prime_\delta(u^\ve_{\delta, 1})\partial_t (u_{\delta, 1}^\ve - u^\ve_{\delta, 2}) , u^\ve_{\delta, 1} - u^\ve_{\delta, 2} \big\rangle_{G^\ve_\tau } \\
  + \big\langle  (b^\prime_\delta(u^\ve_{\delta, 1})- b^\prime_\delta(u^\ve_{\delta, 2}))\,  \partial_t u_{\delta, 2}^\ve,  u^\ve_{\delta, 1} - u^\ve_{\delta, 2} \big \rangle_{G^\ve_\tau} \\
  \geq \frac 12 \delta  \|u^\ve_{\delta, 1}(\tau) - u^\ve_{\delta, 2}(\tau)\|^2_{L^2(G^\ve)}  - C_1\tau^{\frac 1 2} \sup_{(0,\tau)} \|\nabla(u^\ve_{\delta, 1} - u^\ve_{\delta, 2}) \|^2_{L^2(G^\ve)} . 
 \end{aligned}
\end{equation}
 Here we used  the fact that the  continuous embedding $H^1(G^\ve)\subset L^6(G^\ve)$   for $n\leq 3$ and regularity $\partial_t u_{\delta, j}^\ve \in L^2(0,T; L^6(G^\ve))$ and  $u_{\delta, j}^\ve \in L^\infty(0,T; L^6(G^\ve))$, with $j=1,2$,  together with assumptions on $b$, ensure
\begin{equation}\label{estim_b_t}
\begin{aligned}
& \int_{G^\ve_\tau}  |\partial_t b^{\prime}_\delta(u^\ve_{\delta, 1})||u^\ve_{\delta, 1} - u^\ve_{\delta, 2}|^2 dx dt +  \int_{G^\ve_\tau} |b^{\prime}_\delta(u^\ve_{\delta, 1}) - b^{\prime}_\delta(u^\ve_{\delta, 2}) | |  \partial_t u^\ve_{\delta, 2}| |u^\ve_{\delta, 1} - u^\ve_{\delta, 2}| dx dt
 \\ 
 & \leq  C_1 \Big(\int_{G^\ve_\tau}  \big[ |b^{\prime \prime}_\delta(u^\ve_{\delta, 1})| +  |b^{\prime \prime}_\delta(u^\ve_{\delta, 2})|\big]^{\frac 32 }
 \big[ |\partial_t u_{\delta, 1}^\ve|+ |\partial_t u_{\delta, 2}^\ve|\big]^{\frac 32} dx dt \Big)^{\frac 2 3}\Big(\int_{G^\ve_\tau} |u^\ve_{\delta, 1} - u^\ve_{\delta, 2}|^6 dx dt\Big)^{\frac 1 3}\\
 & \leq C_2 \tau ^{\frac 12 } \big(\|\nabla u^\ve_{\delta, 1}\|_{L^\infty(0,\tau; L^2(G^\ve))} + \|\nabla u^\ve_{\delta, 2}\|_{L^\infty(0,\tau; L^2(G^\ve))} +1\big)\, \big(\|\partial_t u^\ve_{\delta, 1}\|_{L^2(G^\ve_\tau)}   \\
 & \hspace{8 cm } + \|\partial_t u^\ve_{\delta, 2}\|_{L^2(G^\ve_\tau)} \big)  \sup_{(0,\tau)}   \|\nabla(u^\ve_{\delta, 1} - u^\ve_{\delta, 2})\|^2_{L^2(G^\ve)}. 
\end{aligned}
\end{equation}
Notice that $u^\ve_{\delta, j}$, with $j=1,2$, satisfies Dirichlet boundary condition and Poincar\'e inequality can be applied. 
Lipschitz continuity of $k$ and regularity assumptions on $\partial_t u^\ve_{\delta,2}$ ensure 
 $$
 \begin{aligned} 
\big | \left \langle A^\ve(x) (k_\delta(u^\ve_{\delta, 1}) - k_\delta(u_{\delta, 2}^\ve))\, \partial_t \nabla u^\ve_{\delta, 2}, \nabla (u^\ve_{\delta, 1}-  u^\ve_{\delta, 2})\right \rangle_{G^\ve_\tau} \big| \leq  C_{1,\tau} \|\nabla(u_{\delta, 1}^\ve - u_{\delta, 2}^\ve)\|^2_{L^2(G^\ve_\tau)} \\ + C_2\tau^{\frac 1 2} \|\partial_t \nabla u_{\delta, 2}^\ve\|^2_{L^2(0, \tau; L^{p}(G^\ve))} \sup_{(0, \tau)} \|u_{\delta, 1}^\ve - u_{\delta, 2}^\ve\|^2_{L^{\frac{2 p}{p - 2}}(G^\ve)}
\\ \leq C_\tau \|\nabla(u_{\delta, 1}^\ve - u_{\delta, 2}^\ve)\|^2_{L^2(G^\ve_\tau)} + \tau^{\frac 1 2} \sup_{(0, \tau)} \|\nabla(u_{\delta, 1}^\ve - u_{\delta, 2}^\ve)\|^2_{L^2(G^\ve)}, 
 \end{aligned} 
 $$
for $p\geq n$.   Using assumptions on $k$ and $P_c$ we also obtain 
 $$
 \begin{aligned} 
& \big |\left \langle A^\ve(x) (k_\delta(u^\ve_{\delta, 1}) P_{c, \delta}(u^\ve_{\delta, 1}) - k(u^\ve_{\delta, 2}) P_{c, \delta}(u^\ve_{\delta, 2})) \nabla u_{\delta, 2}^\ve,  \nabla (u^\ve_{\delta, 1}-u^\ve_{\delta, 2})\right \rangle_{G^\ve_\tau} \big|\\
 & \qquad  \leq  C \|\nabla u_{\delta, 2}^\ve \|^2_{L^\infty(0,\tau;  L^p(G^\ve))} \|\nabla(u_{\delta, 1}^\ve - u_{\delta, 2}^\ve)\|^2_{L^2(G_\tau^\ve)} + \|\nabla(u_{\delta, 1}^\ve - u_{\delta, 2}^\ve)\|^2_{L^2(G_\tau^\ve)}, 
  \end{aligned} 
 $$
for $p \geq n$.  The last two terms  in \eqref{Unique_2} are estimated using Lipschitz continuity of $F^\ve$ and $f^\ve$  and  the trace estimate. 
Then integrating by parts in the first term in \eqref{Unique_2}, using the fact that $k_\delta(u^\ve_{\delta, 1}) \geq \delta >0$,   choosing a sufficiently small $\tau>0$ and applying the Gronwall inequality we obtain 
$$
\sup_{(0,\tau)} \|\nabla(u^\ve_{\delta, 1} - u^\ve_{\delta, 2})\|^2_{L^2(G^\ve)} \leq 0 .
$$
Using the Poincar\'e inequality and iterating over $\tau>0$, which depends on the coefficients in the variational inequality and is independent of a solution of \eqref{main2_regul},    yield $u_{\delta, 1}^\ve(t,x) = u_{\delta, 2}^\ve(t,x)$ a.e. in $(0,T)\times G^\ve$,  and hence the uniqueness of a solution of variational inequality  \eqref{main2_regul}. 
\\
If $k(\xi) = {\rm const}$,  for two solutions $u_{\delta, 1}^\ve$ and  $u_{\delta, 2}^\ve$ of  \eqref{main2_regul} we have  
\begin{equation}\label{Unique_22}
\begin{aligned}
&\frac 12  \int_{G^\ve}  A^\ve(x)  |\nabla (u^\ve_{\delta, 1} -  u^\ve_{\delta, 2})(\tau)|^2  dx    
+ \int_{G^\ve_\tau} A^\ve(x)  P_{c, \delta} (u^\ve_{\delta,1})|\nabla (u_{\delta, 1}^\ve  -  u_{\delta, 2}^\ve)|^2 dx dt
\\
&+\frac 12  \int_{G^\ve} b^\prime_\delta(u^\ve_{\delta, 1})|u^\ve_{\delta, 1}(\tau) - u^\ve_{\delta, 2}(\tau)|^2 dx   +  \int_{G^\ve_\tau} (b^{\prime}_\delta(u^\ve_{\delta, 1}) - b^{\prime}_\delta(u^\ve_{\delta, 2}) )  \partial_t u^\ve_{\delta, 2} (u^\ve_{\delta, 1} - u^\ve_{\delta, 2}) dx dt
\\ 
& - \frac 12 \int_{G^\ve_\tau}  \partial_t b^{\prime}_\delta(u^\ve_{\delta, 1})|u^\ve_{\delta, 1} - u^\ve_{\delta, 2}|^2 dx dt
+ \left \langle A^\ve(x) (P_{c, \delta}(u^\ve_{\delta, 1}) -  P_{c, \delta}(u^\ve_{\delta, 2})) \nabla u_{\delta, 2}^\ve,  \nabla (u^\ve_{\delta, 1}-u^\ve_{\delta, 2})  \right \rangle_{G^\ve_\tau} 
\\
& \leq 
  \langle   F^\ve(t,x, u^\ve_{\delta, 1}) -  F^\ve(t,x, u^\ve_{\delta, 2}),  \nabla    (u^\ve_{\delta, 1} - u^\ve_{\delta, 2})  \rangle_{G^\ve_\tau} -
\ve \langle  f^\ve(t,x,u^\ve_{\delta, 1}) -  f^\ve(t,x,u^\ve_{\delta, 2}),    u^\ve_{\delta, 1} - u_{\delta, 2}^\ve  \rangle_{\Gamma^\ve_\tau}  .
 \end{aligned}
\end{equation}
The   fourth  and fifth terms  on the left-hand side in \eqref{Unique_22} are estimates as in  \eqref{estim_b_t}.  For  the sixth  term on the left-hand side, using Lipschitz continuity of $P_{c, \delta}$ and  regularity assumption on $u_{\delta, 2}^\ve$  we have 
$$
\begin{aligned}
\left \langle A^\ve(x) (P_{c, \delta}(u^\ve_{\delta, 1}) -  P_{c, \delta}(u^\ve_{\delta, 2})) \nabla u_{\delta, 2}^\ve,  \nabla (u^\ve_{\delta, 1}-u^\ve_{\delta, 2})  \right \rangle_{G^\ve_\tau}  \leq C_1 \tau^{\frac 1 2} \sup_{(0,\tau)}\|u^\ve_{\delta, 1} - u^\ve_{\delta, 2} \|^2_{L^{\frac{2p}{p-2}}(G^\ve)}\|\nabla u_{\delta, 2}^\ve\|^2_{L^2(0,\tau; L^p(G^\ve))}  \\ +  
C_{2, \tau}\| \nabla (u^\ve_{\delta, 1}-u^\ve_{\delta, 2})\|^2_{L^2(G^\ve_\tau)}  \leq \tau^{\frac 1 2} \sup_{(0,\tau)}\|\nabla(u^\ve_{\delta, 1} - u^\ve_{\delta, 2}) \|^2_{L^2(G^\ve)} + C_{\tau}\| \nabla (u^\ve_{\delta, 1}-u^\ve_{\delta, 2})\|^2_{L^2(G^\ve_\tau)}  .
 \end{aligned}
 $$
 Lipschitz continuity of $F^\ve$ and $f^\ve$ ensures the corresponding estimates for the terms on the right-hands side of  \eqref{Unique_22}. 
 Combining those estimates, applying Gronwall inequality,  and iterating over $\tau>0$  yield   uniqueness of a solution of  variational inequality  \eqref{main2_regul} if $k ={\rm const}$.  Notice that if both $k$ and $P_c$ are constant the uniqueness result is obtain without  additional regularity assumptions on solutions of  variational inequality~\eqref{main2_regul}.
\end{proof} 
{\bf Remark.}  By extending the $L^p$-theory for parabolic equations to pseudoparabolic equations and variational inequalities  it may be possible to prove higher regularity for  solutions of variational inequality  \eqref{main2_regul}. However this nontrivial analysis will not be considered here and will be the topic of further research. \\

To prove existence of a solution of the original problem \eqref{main2} and to derive macroscopic variational inequality we first derive a priori estimates for solutions of regularised problem   \eqref{main2_regul}  uniformly in $\delta$ and $\ve$. 
\begin{lemma} \label{lem:apriori}
Under Assumption~\ref{assum_1}  and  if  $\beta\geq \lambda > 4+ \alpha$ for $n=3$ and  $\beta\geq \lambda > 3+\alpha + 4/(q-2)$ for $n=2$ and  any $q>2$,  solutions of variational inequality \eqref{main2_regul} are non-negative and  satisfy the following a priori estimates 
\begin{equation}\label{a_priori_1}
\begin{aligned}
&\| (u_\delta^\ve+\delta)^{1+\alpha-\beta}\|_{L^\infty(0,T; L^{1}(G^\ve))} + \|\sqrt{P_{c, \delta}(u^\ve_\delta)} \nabla u^\ve_\delta\|_{L^2((0,T)\times G^\ve)}      \leq C,  \\
 &\|\nabla u^\ve_\delta\|_{L^\infty(0,T; L^2(G^\ve))}+ \|b_\delta(u^\ve_\delta)\|_{L^\infty(0,T; L^2(G^\ve))}  \leq C, \\
  &\|\sqrt{k_\delta (u^\ve_\delta)}\partial_t\nabla u^\ve_\delta\|_{L^2((0,T) \times G^\ve)} +  \|\sqrt{b^\prime_\delta (u^\ve_\delta)}\partial_t u^\ve_\delta\|_{L^2((0,T)\times G^\ve)}   \leq C, \\
  &  \|\partial_t b_\delta(u^\ve_\delta)\|_{L^2(0,T; L^r(G^\ve))}  + \|\nabla \partial_t u^\ve_\delta\|_{L^p((0,T)\times G^\ve)}  \leq C,
 \end{aligned}
\end{equation}
for $1<p<2$ defined in \eqref{eq:p},   $1<r<3/2$ for $n=3$ and $1<r<4/3$ for $n=2$,  and the constant $C>0$  is independent of $\ve$ and $\delta$. 
\end{lemma} 

\begin{proof} [Proof. ]
To show that solutions of  \eqref{main2_regul}  are non-negative we consider $v^\ve_\delta= u^\ve_\delta -\widetilde  h((u^\ve_\delta)^-)$ as a test function in \eqref{main2_regul}, where 
$u^{-}=\min\{ u, 0\}$ and $$\widetilde h(w)=\int_{0}^{w} \frac 1 {k_\delta(\xi)} d\xi.$$ 
Notice that $v^\ve_\delta (t,x)= \kappa_D \geq 0$ on $\partial G$ and  $v^\ve_\delta(t,x)\geq 0$ on $\Gamma^\ve$ for $t \in (0,T)$. The definition of $\widetilde h$  implies that  $\widetilde h((u^\ve_\delta)^{-}) =0$ if $u^\ve_\delta\geq 0$ and  $\widetilde h((u^\ve_\delta)^{-}) <0$ for $u^\ve_\delta < 0$, and   hence   $\widetilde h((u^\ve_\delta)^{-})  =   (u^\ve_\delta)^{-} / k_\delta(\delta)$.
Thus we obtain  
 \begin{equation}\label{main2_posit}
\begin{aligned} 
\langle \partial_t b_\delta(u^\ve_\delta), \widetilde h((u^\ve_\delta)^{-}) \rangle_{G^\ve_\tau}  + \langle  A^\ve(x) (P_{c, \delta}(u^\ve_\delta) \nabla u^\ve_\delta +  \partial_t \nabla u^\ve_\delta) , \nabla(u^\ve_\delta)^{-} \rangle_{G^\ve_\tau}\\
   - \langle   F^\ve(t,x, u^\ve_\delta) , \nabla \widetilde h((u^\ve_\delta)^{-} )\rangle_{G^\ve_\tau}  
 + \ve  \langle  f^\ve(t, x, u^\ve_\delta), \widetilde h((u^\ve_\delta)^{-}) \rangle_{\Gamma^\ve_\tau}     \leq 0, 
\end{aligned} 
\end{equation}
for $\tau \in (0, T]$. Using the  definition of $\widetilde h$ and  properties of $f^\ve$, for the boundary integral we have
$$
  \langle\ve f^\ve(t,x,u^\ve_\delta), \widetilde h((u^\ve_\delta)^{-}) \rangle_{\Gamma^\ve_\tau} =  \langle\ve f^\ve(t,x,u^\ve_\delta), \widetilde h((u^\ve_\delta)^{-})\chi_{u^\ve_\delta \leq 0} \rangle_{\Gamma^\ve_\tau}  \geq 0.  
$$
Assumptions on $F^\ve$  and the boundary conditions on $\partial  G^\ve$  imply
$$
\begin{aligned}
 \langle   F^\ve(t,x, u^\ve_\delta) , \nabla \widetilde h((u^\ve_\delta)^{-} )\rangle_{G^\ve_\tau} = \langle g, \nabla(u^\ve_\delta)^{-} \rangle_{G^\ve_\tau} 
 + \int_0^\tau \int_{G^\ve} \nabla\cdot \ \widetilde H^\ve_\delta(t,x, (u_\delta^\ve)^{-}) \, dx dt   =0, 
 \end{aligned} 
$$
where  $\widetilde H^\ve_\delta(t,x, v) = Q^\ve(t,x)  \int_{0}^v  H(\xi)/ k_\delta(\xi) \, d\xi$. 
Assumptions on $b$, the definition of $\widetilde h$, and the non-negativity of initial data  ensure 
$$
\begin{aligned}
\langle \partial_t b_\delta(u^\ve_\delta),  \widetilde h((u^\ve_\delta)^{-}) \rangle_{G^\ve_\tau} = \langle \partial_t b_\delta((u^\ve_\delta)^{-}),  \widetilde h((u^\ve_\delta)^{-}) \rangle_{G^\ve_\tau}  = \int_{G^\ve} \int_0^{(u^\ve_\delta(\tau))^{-}}\hspace{-0.05 cm }  b^\prime_\delta (\xi) \int_0^\xi \frac {d\eta}{k_\delta(\eta)} d\xi dx \geq 0, 
 \end{aligned}
$$  
for $\tau \in (0, T]$. Then the non-negativity of initial conditions, i.e.\  $u_0(x)\geq 0$  in $G$, and assumptions on $A$  yield  
$$
\sup_{(0,T)} \|\nabla  (u_\delta^\ve)^{-}\|_{L^2(G^\ve)} = 0, 
$$
and using  the non-negativity  of $u_\delta^\ve$ on  $(0,T)\times \partial G^\ve$ we conclude  $u^\ve_\delta(t,x) \geq 0$ a.e.\ in $(0,T)\times G^\ve$.

To derive a priori estimates  in \eqref{a_priori_1}, we  first consider $v^\ve_\delta= u^\ve_\delta - h_\delta(u^\ve_\delta)$ as a test function in \eqref{main2_regul}, where 
$$h_\delta(v) = \theta \int_{\kappa_D}^{v} \frac 1{ k_\delta(\xi)} d\xi \; \; \text{  and  } \; \; \theta = \min_{z\geq \kappa_D } k(z)  >0,  $$  
and  obtain 
\begin{equation}\label{regul_ineq1}
\begin{aligned}
\langle \partial_t b_\delta(u^\ve_\delta) , h_\delta(u^\ve_\delta)  \rangle_{G^\ve_s}+ \theta \langle A^\ve(x)(P_{c, \delta}(u_\delta^\ve) \nabla u_\delta^\ve +  \partial_t \nabla u_\delta^\ve), \nabla u_\delta^\ve  \rangle_{G^\ve_s} \\
-  \langle  F^\ve (t,x, u^\ve_\delta), \nabla h_\delta(u^\ve_\delta) \rangle_{G^\ve_s}   + 
 \langle \ve f^\ve(t,x , u^\ve_\delta),  h_\delta(u^\ve_\delta) \rangle_{\Gamma^\ve_s}  \leq  0
 \end{aligned}
\end{equation}
for  $s\in (0,T]$.  Notice that $h_\delta(v) <0$ for $v<\kappa_D$, $h_\delta(\kappa_D) =0$, and $0<h_\delta(v)\leq v$ for $v > \kappa_D$.  Thus   we obtain  that $v^\ve_\delta(t)\in \mathcal K^\ve$ for $u^\ve_\delta (t)\in \mathcal K^\ve$, since $v^\ve_\delta(t) \geq 0$ on $\Gamma^\ve$ if $u^\ve_\delta(t) \geq 0$ on $\Gamma^\ve$ and $v^\ve_\delta(t)=\kappa_D$ on  $\partial G$ if $u^\ve_\delta(t) = \kappa_D$ on $\partial G$. 

We shall estimate each term in \eqref{regul_ineq1} separately. The boundary integral can be written as
$$
\begin{aligned}
 \langle \ve f^\ve(t,x , u^\ve_\delta),  h_\delta(u^\ve_\delta) \rangle_{\Gamma^\ve_s} =    \langle \ve f^\ve(t,x , u^\ve_\delta),  h_\delta(u^\ve_\delta) \chi_{u^\ve_\delta < \kappa_D} \rangle_{\Gamma^\ve_s} 
  +   \langle \ve f^\ve(t,x , u^\ve_\delta),  h_\delta(u^\ve_\delta) \chi_{u^\ve_\delta \geq \kappa_D} \rangle_{\Gamma^\ve_s}. 
  \end{aligned}
$$
Assumptions on $f^\ve$  imply  
 $$
 \begin{aligned} 
&  \langle \ve f^\ve(t,x , u^\ve_\delta),  h_\delta(u^\ve_\delta) \chi_{u^\ve_\delta \geq \kappa_D} \rangle_{\Gamma^\ve_s}  \geq 0, 
\\
& \big| \langle \ve f^\ve(t,x , u^\ve_\delta),  h_\delta(u^\ve_\delta) \chi_{u^\ve_\delta < \kappa_D} \rangle_{\Gamma^\ve_s} \big| \leq C, 
\end{aligned}
$$
where the constant $C$ is independent of $\delta$ and $\ve$.  To estimate the third term in \eqref{regul_ineq1} we use the  properties of $Q^\ve$ and $H$ and obtain 
$$
 \langle  F^\ve(t,x, u^\ve_\delta), \nabla h_\delta(u^\ve_\delta) \rangle_{G^\ve_s} = \theta \langle g, \nabla u^\ve_\delta \rangle_{G^\ve_s}
 +  \int_0^s\int_{G^\ve} \nabla\cdot \mathcal H_\delta^\ve(t,x, u^\ve_\delta) dx dt , 
 $$
where 
$
\mathcal H^\ve_\delta(t,x, v) = \theta \, Q^\ve(t,x)  \int_{\kappa_D}^v    H(\xi)[ k_\delta(\xi)]^{-1}  d\xi. 
$ 
Using  $Q^\ve(t,x) \cdot \nu =0$ on $\Gamma^\ve$  and $\mathcal H^\ve_\delta(t,x, \kappa_D) =0$ yields  
$$
\int_0^s \int_{G^\ve} \nabla\cdot \mathcal H_\delta^\ve(t,x, u^\ve_\delta) dx dt =\int_0^s \int_{\partial G^\ve} 
 \mathcal H_\delta^\ve(t,x, u^\ve_\delta)  \cdot \nu \, d\gamma_x dt= 0. 
$$
The first term in \eqref{regul_ineq1}  can be write as 
$$
\begin{aligned}
\langle \partial_t b_\delta(u^\ve_\delta) , h_\delta(u^\ve_\delta)  \rangle_{G^\ve_s}  & = \int_{G^\ve_s} \partial_t  \int_{\kappa_D}^{u^\ve_\delta} b_\delta^\prime(\xi) h_\delta(\xi)  \, d \xi dx dt \\ 
& =
\int_{G^\ve}   \int_{\kappa_D}^{u^\ve_\delta(s)} b^\prime_\delta(\xi) h_\delta(\xi)  \, d \xi dx 
- \int_{G^\ve}  \int_{\kappa_D}^{u^\ve_{\delta}(0)} b^\prime_\delta(\xi) h_\delta(\xi)  \, d \xi dx.
\end{aligned}
$$
The definition of $h_\delta$ and properties of function  $b$ ensure  that for $u^\ve_\delta \leq \kappa_D$
$$
\begin{aligned}
\int_{G^\ve}   \int_{\kappa_D}^{u^\ve_\delta(s)} b^\prime_\delta(\xi) h_\delta(\xi)  \, d \xi dx =  \int_{G^\ve}   \int^{\kappa_D}_{u^\ve_\delta(s)} b^\prime_\delta(\xi)  \int^{\kappa_D}_{\xi}  \frac {d\eta}  {k_\delta (\eta)}  \,  d \xi dx 
\geq  C_1 \int_{G^\ve}  |u^\ve_\delta+\delta|^{(1+\alpha - \beta)} dx - C_2,  
\end{aligned}
$$
for $s\in (0, T]$ and  positive constants $C_1$ and $C_2$, which are independent of $\delta$ and $\ve$.  For $u^\ve_\delta > \kappa_D$, the monotonicity of $b$ and nonnegativity of $k$ ensure
\begin{equation*}
\int_{G^\ve}   \int_{\kappa_D}^{u^\ve_\delta(s)} b^\prime_\delta(\xi) h_\delta(\xi)  \, d \xi dx  =   \theta \int_{G^\ve}  \int_{\kappa_D}^{u^\ve_\delta(s)} b^\prime_\delta(\xi)  \int_{\kappa_D}^{\xi} \frac{ 1}{ k_\delta(\eta)} d\eta \, d\xi  \, dx  \geq 0  
\end{equation*}
for $s\in (0, T]$. 
Then integrating in \eqref{regul_ineq1} by parts with respect to time variable  yields 
\begin{equation}
\begin{aligned} 
 \int_{G^\ve} \left[ \int_{\kappa_D}^{u^\ve_\delta(s)}  b^\prime_\delta (\xi)h_\delta(\xi)  d\xi \,   \chi_{u^\ve_\delta \leq \kappa_D} 
 +  \int_{\kappa_D}^{u^\ve_\delta(s)}   b^\prime_\delta (\xi)h_\delta(\xi) d\xi \,  \chi_{u^\ve_\delta \geq \kappa_D} \right] dx  
 +  
 \int_{G^\ve} |\nabla u^\ve_\delta(s) |^2  dx  
 \\  +\int_{G^\ve_s} P_{c,\delta}(u^\ve_\delta) |\nabla  u^\ve_\delta|^2  dx dt 
  \leq  C_1 + C_2 \int_{G^\ve}  \int_{\kappa_D}^{u^\ve_{\delta}(0)} b_\delta^\prime(\xi) h_\delta(\xi)  \, d \xi dx + C_3 \int_{G^\ve} |\nabla u^\ve_\delta(0) |^2  dx 
  \end{aligned} 
\end{equation}
for $s\in (0, T]$,  where the constants $C_j$, with $j=1,2,3$, are independent of $\ve$ and $\delta$.   Hence assumptions on $u_0$ ensure  
\begin{equation}\label{ineq_reg_2}
\begin{aligned}
\sup_{(0,T)} \int_{G^\ve}  |u^\ve_\delta+ \delta|^{1+\alpha-\beta} \chi_{u^\ve_\delta \leq \kappa_D} dx+ \sup_{(0,T)} \int_{G^\ve}  |\nabla u^\ve_\delta|^2dx   + \int_0^T \int_{G^\ve} P_{c,\delta}(u^\ve_\delta) |\nabla  u^\ve_\delta|^2   dx dt \leq  C, 
\end{aligned}
\end{equation}
with a   positive  constant $C$  independent of $\ve$ and $\delta$.

To derive an estimate for $\sqrt{k_{\delta}(u^\ve_\delta)} \partial_t \nabla u^\ve_\delta$ we need to use the  equation with the penalty operator~\eqref{penalty}. Testing  equation \eqref{penalty}  by  $v^\ve =  \partial_t u^\ve_{\delta, \mu}$  yields 
\begin{equation}\label{estim_33}
\begin{aligned} 
\langle \partial_t b_\delta(u^\ve_{\delta,\mu}), \partial_t u^\ve_{\delta, \mu}  \rangle_{G_s^\ve} +   \langle A^\ve(x) k_\delta(u^\ve_{\delta, \mu}) [ P_{c,\delta}(u^\ve_{\delta, \mu}) \nabla u_{\delta, \mu}^\ve +  \partial_t \nabla u_{\delta, \mu}^\ve], \partial_t \nabla u_{\delta, \mu}^\ve \rangle_{G^\ve_s} 
  - \langle  F^\ve(t,x, u^\ve_{\delta, \mu}),  \nabla \partial_t u^\ve_{\delta, \mu} \rangle_{G^\ve_s} \\
  + 
 \langle \ve f^\ve (t,x , u^\ve_{\delta, \mu}),  \partial_t u^\ve_{\delta, \mu} \rangle_{\Gamma^\ve_s} 
   +\frac 1 \mu \int_0^s \langle \mathcal B(u^\ve_{\delta, \mu} - \kappa_D), \partial_t u^\ve_{\delta, \mu} \rangle_{V^\prime, V} dt  =0, 
 \end{aligned} 
\end{equation}
for $s \in (0, T]$. 
Using the property of the  projection operator  \eqref{projection_properties}
  for the  difference quotient of $P_{\mathcal K^\ve} u$ with respect to the time variable  we obtain
$$
0 \leq\frac 1 h\langle J(u-P_{\mathcal K^\ve} u), P_{\mathcal K^\ve} u - P_{\mathcal K^\ve} u(\cdot - h)  \rangle_{V^\prime, V}.
$$  
Then,  the last inequality,  together with  the regularity   $\partial_t u^\ve_{\delta, \mu} \in L^2(0,T; V)$ and  the fact that $u_0, \kappa_D \in \mathcal K^\ve$,   yields
$$
\begin{aligned} 
&\int_0^s \langle \mathcal B(\widetilde u^\ve_{\delta, \mu} ), \partial_t u^\ve_{\delta, \mu} \rangle_{V^\prime, V} dt = 
\lim\limits_{h\to 0}  \sum_{j=1}^N  \big \langle \mathcal B(\widetilde u^\ve_{\delta, \mu}(t_j)), u^\ve_{\delta, \mu}(t_j) -u^\ve_{\delta, \mu}(t_{j-1}) \big  \rangle_{V^\prime, V} 
\\
& \qquad = \lim\limits_{h\to 0}  \sum_{j=1}^N   \Big[ \Big \langle J(\widetilde u^\ve_{\delta, \mu} -  P_{\mathcal K^\ve} \widetilde u^\ve_{\delta, \mu})(t_j) , (\widetilde u^\ve_{\delta, \mu} - P_{\mathcal K^\ve} \widetilde u^\ve_{\delta, \mu})\Big|_{t_{j-1}}^{t_i} \Big \rangle_{V^\prime,V } 
\\ 
& \qquad \qquad + 
\big  \langle J(\widetilde u^\ve_{\delta, \mu}  - P_{\mathcal K^\ve} \widetilde u^\ve_{\delta, \mu} )(t_j) ,   P_{\mathcal K^\ve} \widetilde u^\ve_{\delta, \mu}(t_j) - P_{\mathcal K^\ve} \widetilde u^\ve_{\delta, \mu}(t_{j-1}) \big  \rangle_{V^\prime,V } \Big]
\\
&\qquad  \geq\frac 12  \int_{G^\ve}\Big[  | (\widetilde u^\ve_{\delta, \mu}   - P_{\mathcal K^\ve} \widetilde u^\ve_{\delta, \mu})(s)|^2  + | \nabla (\widetilde u^\ve_{\delta, \mu} - P_{\mathcal K^\ve} \widetilde u^\ve_{\delta, \mu})(s)|^2\Big] dx\geq 0, 
\end{aligned} 
$$
where $\widetilde u^\ve_{\delta, \mu}= u^\ve_{\delta, \mu}-\kappa_D$ and $t_j=jh$ for $j=1,\ldots, N$,  and $N \in \mathbb N$, with  $t_N= Nh=s$. Using assumptions on the functions $k$ and $P_c$ and applying the H\"older inequality  yield
$$
\begin{aligned}
\langle A^\ve(x) k_\delta(u^\ve_{\delta, \mu})  P_{c, \delta}(u^\ve_{\delta, \mu}) \nabla u_{\delta, \mu}^\ve, \partial_t \nabla u_{\delta, \mu}^\ve \rangle_{G^\ve_s}
\leq  \sigma \|\sqrt{k_\delta(u^\ve_{\delta, \mu}) } \partial_t \nabla u_{\delta, \mu}^\ve\|_{L^2(G^\ve_s)} \\
+ C_\sigma \|k_\delta(u^\ve_{\delta, \mu}) P_{c,\delta}(u^\ve_{\delta, \mu})\|_{L^\infty(G^\ve_s)}  
\|\sqrt{P_{c, \delta}(u^\ve_{\delta, \mu})} \nabla u_{\delta, \mu}^\ve \|_{L^2(G^\ve_s)}, 
\end{aligned}
$$
for some $0<\sigma \leq a_0/8$. The boundary term can be written as 
$$
\begin{aligned} 
 \langle \ve f^\ve(t,x , u^\ve_{\delta, \mu}),  \partial_t u^\ve_{\delta, \mu} \rangle_{\Gamma^\ve_s} =  
   \ve \int_{\Gamma^\ve_s}  \partial_t  \int_{\kappa_D}^{ u^\ve_{\delta, \mu}}  f^\ve(t,x , \xi) \, d\xi  d\gamma dt 
   -   \ve \int_{\Gamma^\ve_s} \int_{\kappa_D}^{ u^\ve_{\delta, \mu}} \partial_t  f^\ve(t,x , \xi) \, d\xi  d\gamma dt . 
   \end{aligned} 
$$
Hence   assumptions on $f^\ve$ imply 
$$
\begin{aligned} 
 \left|\langle \ve f^\ve(t,x , u^\ve_{\delta, \mu}),  \partial_t u^\ve_{\delta, \mu}  \rangle_{\Gamma^\ve_s} \right| 
 \leq  \sigma \ve \Big[ \int_{\Gamma^\ve} |u^\ve_{\delta, \mu} (s)|^2 d\gamma +  \int_{\Gamma^\ve_s} |u^\ve_{\delta, \mu}|^2 d\gamma dt \Big]+ C_{\sigma},
  \end{aligned} 
$$ 
 with some  constant  $C_\sigma$ independent of $\mu$, $\ve$ and $\delta$, and an arbitrary fixed $\sigma >0$. 
Then the trace estimate 
 $$
 \ve \|v\|^2_{L^2(\Gamma^\ve)} \leq C\big[ \|v\|^2_{L^2(G^\ve)} + \ve^2 \|\nabla v\|^2_{L^2(G^\ve)} \big],  
 $$
which follows from the definition of $G^\ve$ and $\Gamma^\ve$, the standard trace estimate for  $v \in H^1(Y^\ast)$, and a scaling argument,   combined with the properties of  an extension of $u^\ve_{\delta, \mu}$ from $G^\ve$ into $G$, see Remark~\ref{remark_extension},  and  the Dirichlet boundary condition on $\partial G$, ensures
 $$
  \left|\langle \ve f^\ve(t,x , u^\ve_{\delta, \mu}),  \partial_t u^\ve_{\delta, \mu} \rangle_{\Gamma^\ve_s} \right| 
 \leq \sigma_1 \big[  \|\nabla u^\ve_{\delta, \mu} (s)\|^2_{L^2(G^\ve)} +  \|\nabla u^\ve_{\delta, \mu}\|^2_{L^2(G^\ve_s)}\big]  + C, 
 $$
 with   $s\in (0, T]$.  The assumptions on $F^\ve$  and $k$ and the fact that $\partial_t u^\ve_{\delta, \mu}(t,x) = 0$ on $(0,T)\times \partial G$ yield
 $$
 \begin{aligned} 
& \langle  F^\ve(t,x, u^\ve_{\delta, \mu}),  \nabla \partial_t u^\ve_{\delta, \mu} \rangle_{G^\ve_s} = \langle g \, k_\delta(u^\ve_{\delta, \mu}), \nabla \partial_t u^\ve_{\delta, \mu} \rangle_{G^\ve_s}  -   \langle  Q^\ve(t,x) H^\prime(u^\ve_{\delta, \mu})[b_\delta^\prime(u^\ve_{\delta, \mu})]^{-\frac 12} \nabla u_{\delta, \mu}^\ve,  \sqrt{b^\prime_\delta(u^\ve_{\delta, \mu} )} \partial_t u^\ve_{\delta, \mu} \rangle_{G^\ve_s}.
 \end{aligned} 
 $$
Applying  the H\"older inequality and using assumptions on $H$ and $Q^\ve$ we obtain 
 $$
 \begin{aligned} 
& | \langle  F^\ve(t,x, u^\ve_{\delta, \mu}),  \nabla \partial_t u^\ve_{\delta, \mu} \rangle_{G^\ve_s} |  \leq 
\sigma_1\| \sqrt{k_\delta(u^\ve_{\delta, \mu})} \nabla \partial_t u^\ve_{\delta, \mu} \|^2_{L^2(G^\ve_s)}  
\\ &\qquad  + 
\sigma_2  \| \sqrt{b^\prime_\delta(u^\ve_{\delta, \mu})} \partial_t u^\ve_{\delta, \mu} \|^2_{L^2(G^\ve_s)} 
+ C_1 \| \nabla u_{\delta, \mu}^\ve \|^2_{L^2(G^\ve_s)} + C_2, 
 \end{aligned} 
 $$
for  $0<\sigma_1 \leq a_0/8$, $0<\sigma_2 \leq 1/4$ and constants $C_1, C_2>0$ are  independent of $\mu$,  $\ve$,  and $\delta$. 

Using the estimate for $\nabla u^\ve_{\delta, \mu}$ in $L^\infty(0,T; L^2(G^\ve))$ and $\sqrt{P_{c, \delta}(u^\ve_{\delta, \mu})} \nabla u_{\delta, \mu}^\ve$ in $L^2((0,T)\times G^\ve)$, which can be derived in a similar way as  the corresponding  estimates for $\nabla u^\ve_\delta$ and $\sqrt{P_{c, \delta}(u^\ve_{\delta})} \nabla u_{\delta}^\ve$ in  \eqref{ineq_reg_2} by using estimates for the penalty operator $\mathcal B$ similar to those obtained  in the derivation of inequality \eqref{positive_B},  we obtain  
$$
 \| \sqrt{b^\prime_\delta(u^\ve_{\delta, \mu})} \partial_t u^\ve_{\delta, \mu} \|_{L^2(G^\ve_s)} + \|\sqrt{k_\delta(u^\ve_{\delta, \mu})} \partial_t \nabla u_{\delta, \mu}^\ve\|_{L^2(G^\ve_s)} \leq C, 
$$
for any $s \in (0, T]$ and  a constant $C$ independent of $\mu$, $\ve$ and $\delta$.  Notice that assumptions on $k$ and definition of $\theta$  imply  that 
$u^\ve_{\delta, \mu} - \kappa_D - \theta \int_{\kappa_D}^{u^\ve_{\delta, \mu}} [k_\delta(\xi)]^{-1} d\xi \geq 0$.
 Considering $\mu\to 0$ and using continuity and strict positivity  of $k_\delta$ and $b^\prime_\delta$, together with the strong convergence of $u^\ve_{\delta, \mu}$ in $L^2(G^\ve_T)$, as $\mu \to 0$, and  lower-semicontinuity of a norm,  we obtain the third  estimate in \eqref{a_priori_1}.  
 
  If $b$ is Lipschitz continuous we also have   
$$
\| \partial_t b_\delta(u^\ve_\delta) \|^2_{L^2(G^\ve_T)}\leq  \sup_{(t,x) \in G^\ve_T} |b_\delta^\prime(u^\ve_\delta)| \,  \|\sqrt{b^\prime_\delta (u^\ve_\delta)} \partial_t u^\ve_\delta \|^2_{L^2(G^\ve_T)} \leq  C. 
$$
Otherwise,   we can   consider 
$$
\begin{aligned} 
\| \partial_t b_\delta(u^\ve_\delta) \|_{L^2(0,T; L^r(G^\ve))}=  \|b^\prime_\delta (u^\ve_\delta) \partial_t u^\ve_\delta \|_{L^2(0,T; L^r(G^\ve))} \\ \leq  
\sup_{(0,T)} \|\sqrt{b^\prime_\delta(u^\ve_\delta) }\|^{\frac{2-r}r}_{L^{\frac {2r}{2-r}}(G^\ve)} \|\sqrt{b^\prime_\delta (u^\ve_\delta)} \partial_t u^\ve_\delta \|^2_{L^2(G^\ve_T)}, 
\end{aligned} 
$$
for some $1<r<2$.  Then  the first estimate in \eqref{a_priori_1} for $0\leq u^\ve_\delta(t,x) \leq 1$ and  if  $0<\alpha<1$,  and  assumptions on $b^\prime$ for $u^\ve_\delta(t,x) \geq 1$,  combined with  the uniform  boundedness  of $\|u^\ve_\delta\|_{L^\infty(0,T; H^1(G^\ve))}$, ensure  
$$
\sup_{(0,T)}\big \|\sqrt{b^\prime_\delta(u^\ve_\delta) } \big\|^{\frac{2-r}r}_{L^{\frac {2r}{2-r}}(G^\ve)} \leq C, 
$$
where  $1< r < 3/2$ for $n=3$ and $1<r<4/3$ for $n = 2$. 

From assumptions on $b$ and the estimate for $u^\ve_\delta$ in $L^\infty(0, T; H^1(G^\ve))$, we also obtain  the boundedness of $b_\delta(u^\ve_\delta)$ in $L^\infty(0,T; L^2(G^\ve))$, uniformly in $\ve$ and $\delta$. 

To derive the estimate for $\nabla \partial_t u^\ve_\delta$ in $L^p((0,T)\times G^\ve)$,  with  some $p>1$, we follow the same ideas as in~\cite{Mikelic}. Using   assumptions on  $P_c$  together with  $u^\ve_\delta \geq 0$  we can rewrite 
$$
\sqrt{P_{c, \delta}(u^\ve_\delta)}\nabla u^\ve_\delta = \nabla \Big( \int_0^{u^\ve_{\delta} } \sqrt{P_{c,\delta}(\xi)} d\xi\Big),  
$$
where 
$$
\int_0^{u^\ve_{\delta} } \sqrt{P_{c,\delta}(\xi)} d\xi   = C_1 \left[(u^\ve_\delta + \delta)^{1-\lambda/2} - \delta^{1-\lambda/2}  \right]+ C_2, 
$$
with  some constants $C_1$ and $C_2$ independent of $\ve$ and $\delta$. 
Then the estimate for $P_{c, \delta}(u^\ve_\delta)|\nabla u^\ve_\delta|^2$,  together with  the Dirichlet boundary condition on $\partial G$,  implies that 
$(u^\ve_\delta+\delta)^{1- \lambda/2} \in L^2(0,T; H^1(G^\ve))$.  Considering  an extension $\overline{(u^\ve_\delta+\delta)}^{1-\frac \lambda 2}$ of $(u^\ve_\delta+\delta)^{1-\frac \lambda 2}$ from $G^\ve$ into $G$, see Remark~\ref{remark_extension} applied to $v^\ve = (u^\ve_\delta+\delta)^{1- \frac \lambda 2}$,  we obtain 
$$
\begin{aligned} 
\|\nabla \overline{(u^\ve_\delta+\delta)}^{1- \lambda/2} \|_{L^2((0,T)\times G)} & \leq  C_1\|\nabla (u^\ve_\delta+\delta)^{1- \lambda/2} \|_{L^2((0,T)\times G^\ve)}  \leq C_2, \\
\| (u^\ve_\delta+\delta)^{1- \lambda/2} \|_{L^2((0,T)\times G^\ve)} & \leq \| \overline{(u^\ve_\delta+\delta)}^{1- \lambda/2} \|_{L^2((0,T)\times G)} \\
&\leq  
C_3\|\nabla \overline{(u^\ve_\delta+\delta)}^{1- \lambda/2} \|_{L^2((0,T)\times G)} +C_4 \leq C_5, 
\end{aligned}  
$$
where the constants $C_j$, with $j=1,\ldots,5$, are independent of   $\delta$ and $\ve$.  Notice that the extension  $\overline{(u^\ve_\delta+\delta)}^{1- \lambda/2}$ satisfies the same Dirichlet boundary condition on $\partial G$ as the original function $(u^\ve_\delta+\delta)^{1- \lambda/2}$.   Then the Sobolev embedding theorem ensures 
\begin{equation}\label{embed_11}
\begin{aligned} 
&\|\overline{(u^\ve_\delta+\delta)}^{1- \lambda/2} \|_{L^2(0,T; L^{q_1}(G))} \leq C, \qquad &&  \quad q_1\in (2, + \infty) \, && \text{ for } n =2,\\
&\|\overline{(u^\ve_\delta+\delta)}^{1- \lambda/2} \|_{L^2(0,T; L^{q_2}(G))} \leq C, \qquad &&  \quad q_2 = \frac {2n}{n-2}  \; && \text{ for } n \geq 3, 
\end{aligned} 
\end{equation}
with a constant $C>0$ independent of $\ve$ and $\delta$. 

For  $\theta$ and $\theta_1$ such that $(1-\lambda/2)\theta + (1+ \alpha -\beta)\theta_1 = -\gamma \beta$, where $\gamma >1$ and $\beta$ is as in the assumption on $k$,   we obtain  
\begin{equation}\label{eq:ineq_22}
\begin{aligned} 
& \int_{G^\ve} ({u^\ve_\delta + \delta})^{-\gamma \beta} dx = \int_{G^\ve} ({u^\ve_\delta + \delta})^{(1-\lambda/2)\theta}  (u^\ve_\delta + \delta)^{(1+ \alpha -\beta)\theta_1} dx \\
& \leq \Big(\int_{G^\ve} ({u^\ve_\delta + \delta})^{(1- \lambda/2) p} dx \Big)^{\theta/p} \Big(\int_{G^\ve} ({u^\ve_\delta + \delta})^{(1+ \alpha- \beta) \theta_1 p_1} dx \Big)^{1/p_1}\\
& \leq \Big(\int_{G} (\overline{u^\ve_\delta + \delta})^{(1- \lambda/2) p} dx \Big)^{\theta/p} \Big(\int_{G^\ve} ({u^\ve_\delta + \delta})^{(1+ \alpha- \beta) \theta_1 p_1} dx \Big)^{1/p_1}.
\end{aligned} 
\end{equation} 
For $n=3$ we have $p=6$ and $p_1 = 6/(6-\theta)$.  Then the estimate for $(u^\ve_\delta+\delta)^{(1+\alpha - \beta)}$ in $L^1((0,T)\times G^\ve)$ yields  $\theta_1 = 1- \theta/6$ and the integrability of  $\overline{(u^\ve_\delta+\delta)}^{1- \lambda/2}$  with respect to the time variable implies $\theta =2$.  
Hence $-\gamma \beta = 2- \lambda + \frac 2 3(1+ \alpha - \beta)$ and in order to ensure that   $\gamma>1$ we require 
\begin{equation}\label{ineq1}
- \frac 1{\beta} \left( \frac 8 3 + \frac 2 3\alpha - \lambda - \frac 2 3 \beta\right) >1 \quad 	\Longleftrightarrow \quad  \frac 8 3 + \frac 2 3 \alpha + \frac \beta 3 < \lambda. 
\end{equation}
If $n=2$  the H\"older exponents in    \eqref{eq:ineq_22} are $p= q_1/\theta$ and $1/p_1= 1 - \theta/q_1$, for any $q_1 > 2$. 
Thus we obtain $\theta=2$, $\theta_1 = 1- 2/q_1$ and 
\begin{equation}\label{ineq2}
- \gamma\beta = (2-\lambda) + (1+ \alpha - \beta) (1- 2/q_1)  \quad 	\text{ and }  \; \; \gamma>1    \;  	\Longleftrightarrow \;     3- \frac 2 {q_1} + \alpha\big(1- \frac 2{q_1}\big) + \frac 2 {q_1} \beta <  \lambda .
\end{equation}
Then, combining  the third estimate in \eqref{a_priori_1},  \eqref{embed_11} and  \eqref{eq:ineq_22},     we obtain   the following estimate
\begin{equation}\label{estim_pseudo_parab_term}
\begin{aligned} 
&\int_0^T \int_{G^\ve} |\nabla \partial_t u^\ve_\delta|^p dx dt = \int_0^T  \int_{G^\ve}  |k_\delta(u^\ve_\delta)^{\frac 1 2}\nabla \partial_t u^\ve_\delta|^p  | k_\delta(u^\ve_\delta)|^{-\frac p 2} dx dt\\
&\leq  \Big(\int_0^T  \int_{G^\ve} k_\delta(u^\ve_\delta)  |\nabla \partial_t u^\ve_\delta|^2 dx dt \Big)^{\frac p 2} 
\Big(\int_0^T  \int_{G^\ve} |k_\delta(u^\ve_\delta)|^{-\frac p{2-p}} dx dt\Big)^{1-\frac p2}\\
& \leq C_1 \Big(\int_0^T  \int_{G^\ve} |k_\delta(u^\ve_\delta)|^{-\frac p{2-p}} dx dt\Big)^{1-\frac p2}
\end{aligned} 
\end{equation}
for some $1<p<2$. Assumptions on  $k$,  conditions on  $\alpha$, $\beta$ and $\lambda$, specified in the formulation of the lemma,  and the first estimate in \eqref{a_priori_1} ensure that there exists such $p=p(\beta, \lambda, \alpha,n)>1$   that 
$$
 \| k_\delta(u_\delta^\ve)^{-p/(2-p)} \| _{L^1((0,T)\times G^\ve) } \leq C_2, 
 $$ 
 where $C_2$ is independent of $\ve$ and $\delta$ and  the exponent $p$ is defined as   
\begin{equation}\label{eq:p} 
\begin{aligned} 
& p = \frac{ 2(3 \lambda + 2 \beta - 2\alpha - 8)}{3\lambda + 5 \beta - 2 \alpha - 8}  &&  \text{ for }  n =3   \text{ and } \beta\geq \lambda > 4 +  \alpha,   \\
& p =  \frac{ 2[2(1+ \alpha - \beta) + q_1 ( \lambda+ \beta - 3 - \alpha)]}{ 2(1+ \alpha - \beta) + q_1( \lambda+2 \beta  -3- \alpha)} &&  \text{ for } n=2,  \text{ any }  q_1> 2,   \text{ and }  \beta \geq \lambda > 3 + \alpha + 4/(q_1 - 2), 
\end{aligned}
\end{equation}
and additionally  inequalities in  \eqref{ineq1} and \eqref{ineq2} are satisfied. This implies the  last estimate in \eqref{a_priori_1}. 
\end{proof}
 
\begin{remark}\label{remark_extension} 
 To ensure that in the derivation of a priori estimates the embedding and Poincar\'e constants are independent of $\ve$,  we considered an extension of $u^\ve_\delta$ and of $(u^\ve_\delta+\delta)^{1- \lambda/2}$   from $G^\ve$ to $G$ with the following properties: 
 There exists an extension $\overline v^\ve$  of  $v^\ve$   from $L^p(0,T; W^{1,p}(G^\ve))$  into $L^p(0, T; W^{1,p}(G))$ such that
\begin{equation}\label{estim_ext_1}
\| \overline v^\ve \|_{L^p(G_T)} \leq  C \|v^\ve \|_{L^p(G^\ve_{T})}, \quad \| \nabla \overline v^\ve \|_{L^p(G_T)} \leq  C \|\nabla v^\ve \|_{L^p(G^\ve_{T})},
\end{equation}
where  $1\leq p < \infty$ and the constant $C>0$ is independent of $\ve$. 
The existence of an extension $\overline v^\ve$ satis\-fying estimates \eqref{estim_ext_1} follows from the assumptions on the geometry of $G^\ve$  and a standard extension operator, see  e.g.\ \cite{Acerbi, CiorPaulin99}.
\end{remark}

A priori estimates \eqref{a_priori_1}  ensure the following convergence results for a subsequence of  $\{u^\ve_\delta\}$ as $\delta \to 0$:
\begin{lemma} \label{lem:conv_delta}
Under assumptions in Lemma~\ref{lem:apriori},  there exists a function $u^\ve \in L^2(0,T; H^1(G^\ve))$, with $\partial_t u^\ve \in L^p(0,T; W^{1, p}(G^\ve))$,    such that, up to a subsequence, 
\begin{equation}\label{convergence:u_ve_delta1}
\begin{aligned}
& u^\ve_\delta  \to u^\ve &&  \text{strongly in } L^2(0,T; H^\sigma(G^\ve))  \text{ and in } L^{r_1}((0,T)\times G^\ve) \; \text{ for } 1< r_1< 6, \\
 &  b_\delta(u^\ve_\delta)  \to  b(u^\ve) \; \; &&  \text{strongly in } L^{r_2}((0,T)\times G^\ve) \;   \text{ for } 1< r_2< 2, \\
  &  k_\delta(u^\ve_\delta)  \to  k(u^\ve) && \text{strongly  in } L^q((0,T)\times G^\ve) \; \text{ for } 1<q<  \infty,  \\
   & b_\delta(u^\ve_\delta) \rightharpoonup  b(u^\ve) && \text{weakly-$\ast$ in } L^\infty(0,T; L^2(G^\ve)),  \\
   & u^\ve_\delta \rightharpoonup  u^\ve && \text{weakly-$\ast$ in } L^\infty(0,T; H^1(G^\ve)),  
\end{aligned}
\end{equation}
where $1/2 < \sigma < 1$, and 
\begin{equation}\label{convergence:u_ve_delta2}
\begin{aligned}
  &  \partial_t b_\delta(u^\ve_\delta)  \rightharpoonup   \partial_t b(u^\ve)  && \text{weakly  in } L^2(0,T; L^r(G^\ve)), \\
 & \partial_t u^\ve_\delta \rightharpoonup  \partial_t u^\ve && \text{weakly in } L^p(0,T; W^{1,p}(G^\ve)),\\
 & \sqrt{k_\delta(u^\ve_\delta)} \nabla \partial_t u^\ve_\delta \rightharpoonup \sqrt{k(u^\ve)} \nabla\partial_t  u^\ve && \text{weakly in }  L^2((0,T)\times G^\ve), \\
  & \sqrt{k_\delta(u^\ve_\delta) P_{c,\delta}(u^\ve_\delta)} \nabla  u^\ve_\delta \rightharpoonup \sqrt{k(u^\ve)P_c(u^\ve)} \nabla   u^\ve  && \text{weakly in }  L^2((0,T)\times G^\ve),
\end{aligned}
\end{equation}
as $\delta \to 0$,    where  $1<p<2$ is defined in \eqref{eq:p},  $1< r < 3/2$ for $n=3$ and $1<r<4/3$ for $n=2$.  Due to the lower semicontinuity of a norm we also have 
\begin{equation}\label{estim:u_ve}
\begin{aligned} 
\|\nabla u^\ve\|_{L^\infty(0,T; L^2(G^\ve))} + \|\sqrt{k(u^\ve)}\partial_t\nabla u^\ve\|_{L^2(G^\ve_T)} + \|\nabla \partial_t u^\ve\|_{L^p(G^\ve_T)} \qquad  \\
+ \|b(u^\ve)\|_{L^\infty(0, T; L^2(G^\ve))}  +\|\partial_t b(u^\ve)\|_{L^2(0, T; L^r(G^\ve))}  \leq C, 
\end{aligned} 
\end{equation}
with a constant $C>0$ independent of $\ve$, and $u^\ve(t,x) \geq 0$ in $(0,T)\times G^\ve$. 
\end{lemma} 
\begin{proof}
 Weak-$\ast$ convergence of $u^\ve_\delta$ in $L^\infty(0,T; H^1(G^\ve))$  and weak convergence of $\partial_t u^\ve_\delta$  in $L^p(0,T; W^{1,p}(G^\ve))$ 
 follow directly from the a priori estimates \eqref{a_priori_1}, combined with the Dirichlet boundary condition on $\partial G$ and the Poincar\'e inequality.  Then   using  Lions-Aubin compactness lemma~\cite{Lions} and the fact that embeddings $H^1(G^\ve) \subset H^\sigma(G^\ve)$ for $1/2<\sigma<1$ and  $H^1(G^\ve) \subset L^{r_1}(G^\ve)$ for $1\leq  r_1< 6$ are compact, we obtain the strong convergence of $u^\ve_\delta$ in $L^2(0, T; H^\sigma(G^\ve))$ and  in $L^{r_1}((0,T)\times G^\ve)$.  
 
 Continuity of $b_\delta$, $P_{c, \delta}$ and $k_\delta$  and the strong convergence of $u^\ve_\delta$ imply point-wise convergence 
 $b_\delta(u^\ve_\delta) \to b(u^\ve)$,  $k_\delta(u^\ve_\delta) \to k(u^\ve)$,  $k_\delta(u^\ve_\delta)P_{c,\delta}(u^\ve_\delta) \to k(u^\ve)P_c(u^\ve)$ a.e.\ in $(0,T)\times G^\ve$ as $\delta \to 0$. 
 Assumptions on $b$ yield  $ \|b_\delta(u^\ve_\delta)\|_{L^{r_2}(G^\ve_T)} \leq C_1(1+\|u^\ve_\delta\|^3_{L^{3r_2}(G^\ve_T)} )$, where $3\leq 3r_2<6$. 
 Then the strong convergence of $u^\ve_\delta$ together with the Lebesgue dominated convergence theorem implies the strong convergence of $b(u^\ve_\delta)$ in $L^{r_2}(G^\ve_T)$ for $1< r_2<2$.   Assumptions on functions  $k$ and $P_c$, stated  in Assumption~\ref{assum_1}, ensure  that 
 $|k_\delta(u^\ve_\delta)| \leq C$ and $| k_\delta(u^\ve_\delta) P_{c,\delta}(u^\ve_\delta)| \leq C$ a.e.\ in $G^\ve_T$  independently of $\delta$. Then applying  the Lebesgue dominated convergence theorem implies strong convergence of  $k_\delta(u^\ve_\delta)$  and   $k_\delta(u^\ve_\delta) P_{c,\delta}(u^\ve_\delta)$  in $L^q((0,T)\times G^\ve)$ for any $1< q < \infty$. 
 
 Estimates for $\partial_t b_\delta(u^\ve_\delta)$ together with the convergence  $b_\delta(u^\ve_\delta) \to b(u^\ve)$ in $L^{r_2}(G^\ve_T)$ ensure  weak convergence of 
 $\partial_t b_\delta(u^\ve_\delta) \rightharpoonup   \partial_t b(u^\ve)$  in $L^2(0,T; L^r(G^\ve))$. 
 Weak convergence    $\partial_t u^\ve_\delta$  in $L^p(0,T; W^{1,p}(G^\ve))$  and strong convergence and boundedness  of $k_\delta(u^\ve_\delta)$ ensure weak convergence of $\sqrt{k_\delta(u^\ve_\delta)}\partial_t \nabla u^\ve_\delta  \rightharpoonup \sqrt{k(u^\ve)}\partial_t \nabla u^\ve$ in $L^{p_1}(G^\ve_T)$ for $1<p_1<p$, as $\delta \to 0$.  A priori estimates \eqref{a_priori_1} imply $\sqrt{k_\delta(u^\ve_\delta)}\partial_t \nabla u^\ve_\delta  \rightharpoonup  w$ in $L^2(G^\ve_T)$. Hence $w = \sqrt{k(u^\ve)}\partial_t \nabla u^\ve \in L^2(G^\ve_T)$. Similar arguments  imply the last convergence in \eqref{convergence:u_ve_delta2}. 
\end{proof}

\begin{theorem}
Under assumptions in Lemma~\ref{lem:apriori},  for every fixed $\ve >0$ there exists a  nonnegative solution  of  variational inequality \eqref{main2}.  
 If  $k(\xi)$ is non-degenerate,  $\partial_t u^\ve \in L^2(0, T; W^{1, p_2}(G^\ve))$  and $u_0 \in W^{1, p_2}(G^\ve)$ for $p_2 >n$, or if  $k(\xi)={\rm const}$, $P_c(\xi)$ is Lipschitz continuous for $\xi \geq 0$ and $u^\ve \in L^2(0, T; W^{1,p_2}(G^\ve))$,  then solution of \eqref{main2} is unique. 
\end{theorem}
\begin{proof}
Using  the  convergence results in Lemma~\ref{lem:conv_delta}, together with   assumptions on  $k$, $P_c$,  $b$,  $H$, $f_0$, and $f_1$,  stated  in Assumption~\ref{assum_1},   and taking $\delta\to 0$ in the regularised problem \eqref{main2_regul},  we obtain that $u^\ve$ satisfies  variational inequality~\eqref{main2}.  The regularity of $u^\ve$ implies  $u^\ve\in C([0,T]; L^2(G^\ve))$ and $u^\ve(t) \to u_0$ in $L^2(G^\ve)$ as $t \to 0$.  The weak convergence of $u^\ve_\delta$ in $L^2(0,T; H^1(G^\ve))$ and non-negativity of $u^\ve_\delta$ in $G^\ve_T$, together with $u^\ve_\delta \in \mathcal K^\ve$,  ensure that   $u^\ve(t,x)\geq 0$ in $(0,T)\times G^\ve$ and on $(0,T)\times\Gamma^\ve$, as well as $u^\ve (t,x)= \kappa_D$ on $(0,T)\times \partial G$. Hence $u^\ve(t) \in \mathcal K^\ve$ for $t\in [0, T]$.   

The proof of the uniqueness result in the case $k$ is nondegenerate or $k(\xi) ={\rm const}$  for $\xi \geq 0$ follows the same steps as the corresponding proof for the regularised problem  \eqref{main2_regul} in Lemma~\ref{Lemma_existence_regular}.
\end{proof}

 \section{Derivation of  macroscopic  obstacle problem} \label{section2}
Using estimates \eqref{estim:u_ve}  and compactness theorems for the two-scale convergence, see e.g.\  \cite{Allaire,  Neuss,  Nguetseng} or Appendix for more details, we obtain the following conver\-gen\-ce results  for a subsequence of the sequence $\{u^\ve\}$ of  solutions  of the microscopic problem~\eqref{main2},  as $\ve \to 0$. 
 \begin{lemma} 
Under assumptions in Lemma~\ref{lem:apriori},  there exist functions  $u \in L^2(0,T; H^1(G))$ and $w\in L^2(G_T; H^1_{\rm per} (Y^\ast)/\mathbb R)$,  with  $\partial_t u \in L^p(0,T; W^{1, p}(G))$  and $\partial_t w \in L^p(G_T; W^{1,p}_{\rm per} (Y^\ast)/\mathbb R)$, such that, up to a subsequence, 
\begin{equation}\label{convergence1}
\begin{aligned}
& u^\ve  \to u \quad &&  \text{strongly in } L^{r_1}((0,T)\times G) \; \; \text{ for } 1< r_1<6, \\
&   b(u^\ve)  \to  b(u)  \quad &&  \text{strongly in } L^{r_2}((0,T)\times G) \; \; \text{ for } 1< r_2<2, \\
 &  k(u^\ve)  \to  k(u) \quad && \text{strongly  in } L^q((0,T)\times G) \; \; \; \text{ for } 1< q<\infty, \\
  & b(u^\ve) \rightharpoonup b(u) && \text{weakly-$\ast$ in } L^\infty(0,T; L^2(G)),  \\
& \partial_t b(u^\ve) \rightharpoonup \partial_t b(u) && \text{weakly in } L^2(0,T; L^r(G)), 
\end{aligned}
\end{equation}
for    $1< r<3/2$ for $n=3$ and $1<r<4/3$ for $n=2$, where  $u^\ve$ is identified with its extension, as in Remark~\ref{remark_extension},   and 
\begin{equation}\label{convergence2}
\begin{aligned}
 & \nabla u^\ve \rightharpoonup \nabla u + \nabla_y w && \text{two-scale} ,  \\
  &  \nabla \partial_t u^\ve \rightharpoonup   \nabla\partial_t  u + \nabla_y \partial_t w && \text{two-scale}, \\
 & k(u^\ve) \nabla \partial_t u^\ve \rightharpoonup k(u) (\nabla\partial_t  u + \nabla_y \partial_t w) && \text{two-scale}, \\
  & k(u^\ve) P_c(u^\ve) \nabla  u^\ve \rightharpoonup k(u)P_c(u) (\nabla   u + \nabla_y  w) \; \; && \text{two-scale}, \\
 & \ve \|u^\ve\|^2_{L^2((0,T)\times \Gamma^\ve)}  \to |Y|^{-1}\|u\|^2_{L^2((0,T)\times G\times \Gamma)},  
\end{aligned}
\end{equation}
as  $\ve \to 0$,   where  exponent $p$ is defined  in \eqref{eq:p}.  
\end{lemma} 
\begin{proof}[Proof. ]
The estimate for $\nabla \partial_t u^\ve$  in \eqref{estim:u_ve}, combined with the Dirichlet boundary condition on $\partial G$ and the Poincar\'e and Sobolev inequalities,  ensures that $\partial_t u^\ve$ and its extension $\partial_t \overline u^\ve$, see Remark~\ref{remark_extension}, satisfy the following estimate
$$
\begin{aligned}
& \|\partial_t u^\ve\|_{L^p(0,T; W^{1,p}(G^\ve))}   +  \|\partial_t \overline u^\ve\|_{L^p(0,T; W^{1,p}(G))} + \|\partial_t u^\ve\|_{L^{p}(0,T; L^{q_2}(G^\ve))} + \|\partial_t \overline u^\ve\|_{L^{p}(0,T; L^{q_2}(G))}  \leq C ,
\end{aligned} 
$$
for $1<p<2$ as in \eqref{eq:p},  $q_2=np/(n-p)$, and  a constant $C>0$ independent of $\ve$. 
Then using Lions-Aubin compactness lemma \cite{Lions} we obtain    strong convergence of  $u^\ve$  in $L^{r_1}((0,T)\times G)$, for $1<r_1<6$. Strong convergence of $u^\ve$, continuity of $k$ and $b$,  bounded\-ness of  $k(u^\ve)$, and estimates for  $b(u^\ve)$   and  $\partial_t b(u^\ve)$   ensure the  strong convergence of  $\{k(u^\ve)\}$ and $\{b(u^\ve)\}$ and weak convergence of $\{\partial_t b(u^\ve)\}$. 
A priori estimates  \eqref{estim:u_ve},  the strong convergence of $u^\ve$,  continuity and boundedness of $k(\xi) $ and $k(\xi)P_c(\xi)$ for $\xi \geq 0$, together with   the compactness theorems for the two-scale convergence, see e.g.\  \cite{Allaire, Neuss, Nguetseng},  imply  the first four convergence results  in  \eqref{convergence2}. 
The last convergence in \eqref{convergence2} follows from  the compactness of the embedding  $H^1(G)\subset H^{\sigma}(G)$ for $1/2 < \sigma <1$  and the estimate 
$$
\ve \|v\|^2_{L^2(\Gamma^\ve)} \leq C \|v\|^2_{H^{\sigma}(G^\ve)}  \qquad   \text{ for } \; \;  \sigma >1/2, 
$$
with  a constant $C>0$  independent of $\ve$, see e.g.\  \cite{Ptashnyk_3} for the proof. 
\end{proof}

\begin{theorem}\label{th:macro} 
Under assumptions in Lemma~\ref{lem:apriori},  a subsequence   of $\{ u^\ve\}$, denoted again by $\{ u^\ve\}$,  where $u^\ve$ are solutions  of  problem \eqref{main2}, 
convergences to a function  $u \in \kappa_D+ L^2(0,T; H^1_0(G))$,  with  $\partial_t u \in L^p(0,T; W^{1,p}(G))$,  $\sqrt{k(u) } \partial_t \nabla u \in L^2(G_T)$,  $\partial_t b(u) \in L^2(0,T;  L^r(G))$, where  $1<r < 3/2$ for $n=3$ and $1< r < 4/3$  for $n=2$, and   $p>1$ is defined in \eqref{eq:p}, and $u(t) \in \mathcal K$ for $t \in [0,T]$,  satis\-fying  macroscopic variational inequality 
\begin{equation} \label{macro_main}
\begin{aligned} 
\langle \partial_t b(u), v - u \rangle_{G_T}   +\big \langle A_{\rm hom} k(u) [ P_c(u)  \nabla u +   \partial_t \nabla u] , \nabla(v - u) \big\rangle_{G_T} \quad \\   -
\langle   F_{\rm hom}(t,x,u) ,  \nabla (v - u) \rangle_{G_T}   + 
\langle  f_{\rm hom}(t, u), v - u \rangle_{G_T}  \geq 0
\end{aligned} 
 \end{equation}
 for $v- \kappa_D \in L^2(0,T; H^1_0(G))$, with  $v(t) \in \mathcal K$, where $ \mathcal K$ is defined in   \eqref{K_macro}, 
 $$
 \begin{aligned}
 F_{\rm hom}(t,x,u)  & = \dashint_{Y^\ast} Q(t,x,y)\,  dy \, H(u) + k(u) g,  \\
 f_{\rm hom}(t, u) & = \dashint_{Y^\ast} f_0(t,y)\, dy\,  f_1(u), 
 \end{aligned}
 $$
 and matrix $A_{\rm hom}$ is defined  in \eqref{matrix_A}.
 
 If $k(\xi) = {\rm const}$, $P_c(\xi)$ is Lipschitz continuous for $\xi \geq 0$,  and $u \in L^2(0, T; W^{1,p_2}(G))$ for $p_2>n$ or if $k(u)\geq \delta >0$ for $u\geq 0$, $\partial_t u \in L^2(0, T; W^{1,p_2}(G))$ and $u_0 \in W^{1, p_2}(G)$,  then variational inequality \eqref{macro_main} has a unique solution and the whole sequence of microscopic solutions $\{ u^\ve\}$ converges to the solution of  \eqref{macro_main}.  
\end{theorem} 

\begin{proof}[Proof.] To derive  macroscopic inequality  \eqref{macro_main} we consider  
$$v^\ve(t,x)= u^\ve(t,x) + \phi(t,x) + \sigma(\ve) \varphi(t,x) +  \ve \psi(t,x, x/\ve) $$ 
as a test function in \eqref{main2}, where  $\psi \in C^1_0(G_T, C^1_{\rm per} (Y))$,  $\phi, \varphi \in H^1_0((0,T)\times G)$,  with  $\phi(t,x) + u(t,x) \geq 0$ and  $\varphi(t,x)\geq 0$  in $(0,T)\times G$, and  $\sigma(\ve) \to 0$ as $\ve \to 0$. Notice that since $u^\ve \to u$  strongly two-scale on $(0,T)\times \Gamma^\ve$ as $\ve \to 0$, there exist such functions $\varphi$  and $\sigma(\ve) >0$  that $v^\ve(t,x) \geq 0 $ on $(0,T)\times \Gamma^\ve$ for sufficiently small $\ve>0$. We also have that $v^\ve(t,x) = \kappa_D$ on $(0,T)\times\partial G$.  Then using the   convergence results in \eqref{convergence1} and \eqref{convergence2} and taking  in \eqref{main2} the limit as $\ve \to 0$  yield
\begin{eqnarray}
 |Y^\ast|  \int_{G_T} \partial_t b(u)   \phi \, dx dt + \int_{G_T}   \int_{Y^\ast}   A(y) k(u)\big[\partial_t ( \nabla u + \nabla_y w)  +  P_c(u)  (\nabla  u + \nabla_y w)\big]\big( \nabla \phi+ \nabla_y \psi\big) dy dx dt  \nonumber\\
 - \int_{G_T}  \int_{Y^\ast}  F(t,x, y, u) (\nabla \phi  + \nabla_y \psi) dy dx dt      
    + \int_{G_T}  \int_{\Gamma}  f(t, y, u) \, \phi \, d\gamma_y dx dt \geq 0. \nonumber
\end{eqnarray}  
Assumptions on $F^\ve$, i.e.\ $\nabla \cdot Q^\ve(t,x) =0$ in $G^\ve_T$ and $Q^\ve(t,x)\cdot \nu =0$ on $\Gamma^\ve_T$, which imply that $\nabla_y \cdot Q(t,x,y) =0$ in $G_T\times Y^\ast$,   $Q(t,x,y)\cdot \nu =0$ on  $G_T\times \Gamma$,  and $Q$ is $Y$-periodic, and the fact that $u$ is independent of $y$ ensure 
$$
 \int_{G_T}\int_{Y^\ast}  F(t,x,y, u)  \nabla_y \psi \,  dy dx dt =0.
$$
By choosing $\phi=0$ and $\psi=0$, respectively, we obtain 
 \begin{equation}\label{macro_ineq_w}
  \int_{G_T}  \int_{Y^\ast}  A(y)  k(u) \left [ \partial_t ( \nabla u + \nabla_y w)  + P_c(u)  (\nabla  u + \nabla_y w)\right]  \nabla_y \psi  \, dy  dx dt \geq 0 
\end{equation}
and 
 \begin{equation}\label{macro_ineq}
\begin{aligned} 
 \int_{G_T}\partial_t  b(u) \,  \phi \, dx dt  + \int_{G_T}\dashint_{Y^\ast}  A(y) k(u) \left [ \partial_t ( \nabla u + \nabla_y w)  +  P_c(u)  (\nabla  u + \nabla_y w)\right] \nabla \phi \, dy dx dt  \qquad 
\\  - \int_{G_T}\dashint_{Y^\ast} F(t, x, y, u)dy \, \nabla \phi  \,  dx dt   + \int_{G_T}\frac 1 { |Y^\ast|}\int_{\Gamma}  f(t, y, u) \, d\gamma_y \,  \phi \, dx dt \geq 0. 
\end{aligned} 
\end{equation} 
Considering   $\pm \psi$ in \eqref{macro_ineq_w} yields
\begin{equation*}\label{unit_cell_1}
\begin{aligned} 
&   \int_{G_T} \int_{Y^\ast} A(y) k(u)\left [ \partial_t ( \nabla u + \nabla_y w)  + P_c(u) (\nabla  u + \nabla_y w)\right]  \nabla_y \psi  \, dy dx dt  = 0, 
\end{aligned} 
\end{equation*} 
for all $\psi \in C^1_0(G_T; C^1_{\rm per}(Y))$. For a give $u\in L^2(0,T; H^1(G))$,  the last equation  is a pseudoparabolic  equa\-tion for $w$ with respect to microscopic  variables  $y$:  
\begin{equation}\label{unit_cell_11}
\begin{aligned} 
  \nabla_y\cdot \big(A(y) k(u) [\partial_t (\nabla u+ \nabla_y w)  + P_c(u)(\nabla u+  \nabla_y w)] \big)
 =  0 \; \; \; & \text{ in } \; Y^\ast_T, \\
A(y) k(u) [\partial_t (\nabla u+ \nabla_y w)  + P_c(u)(\nabla u+  \nabla_y w)]  \cdot \nu = 0  \; \; \;  & \text{ on } \Gamma_T, \\
  w \qquad  Y-\text{periodic}, 
\end{aligned} 
\end{equation} 
for  $x\in G$, where $Y^\ast_T = (0,T)\times Y^\ast$. Using a regularisation of $k$ and $P_c$, in a similar way as for \eqref{main2},  we can show the existence of a solution of problem~\eqref{unit_cell_11},  see also the existence proof for \eqref{macro_unit22}  in Lemma~\ref{exist_unit_cell}.  To prove the existence of a solution of \eqref{unit_cell_11}, with regularized  $k$ and $P_c$,   we apply the Rothe method, use the Lax-Milgram theorem for the resulting linear elliptic problem, and  consider $w\, [k(u+\delta)]^{-1}$ and $\partial_t w$ as test functions to derive the corresponding a priori estimates.  We also use the fact that 
$\nabla u \in L^2((0, T)\times G)$,  $k(u) \partial_t \nabla u \in L^2((0,T)\times G)$, and $k(u) P_c(u)$  is bounded. 
Considering the equation for the difference of two solutions $w_1$ and $w_2$ of \eqref{unit_cell_11},  taking $\psi = (w_1 - w_2)\, [k(u+\delta)]^{-1}$, with $\delta>0$,  as a test function, using  assumptions on $A$, and letting $\delta \to 0$, yield 
$$
\|\nabla_y(w_1 - w_2)\|_{L^\infty(0, T; L^2(G\times Y^\ast))} =0.
$$
 Hence a solution of \eqref{unit_cell_11} is defined uniquely up to an additive function independent of $y$. The structure of  \eqref{unit_cell_11} suggests  that  $w$ is of the form 
\begin{equation}\label{form_w}
w(t,x,y)= \sum_{j=1}^n \partial_{x_j}  u(t,x)\,  \omega^j (y) + \overline w(t,x),  
\end{equation}
where $\omega^j$, for $j=1,\ldots, n$,  satisfy   the following `unit cell' problems 
\begin{equation}\label{macro_unit1}
\begin{aligned} 
& {\rm div}_y ( A(y) (\nabla_y \omega^j + e_j)) = 0 \; \;  &&  \text{ in } Y^\ast,  &&  \int_{Y^\ast} \omega^j(y) dy =0, \\
&  A(y) (\nabla_y \omega^j + e_j)\cdot \nu = 0  &&\text{ on } \Gamma, \; \;  && \omega^j  \;  \; Y-\text{periodic}, 
\end{aligned} 
\end{equation}
with $\{e_j\}_{j=1, \ldots, n}$ being the standard  basis  of $\mathbb R^n$. Notice that the well-posedness of \eqref{macro_unit1} follows directly  from the assumptions on $A$ in Assumption~\ref{assum_1}.  \\
Substituting expression \eqref{form_w} for $w$  into \eqref{macro_ineq}  determines the  matrix  $A_{\rm hom} = (A_{\rm hom}^{ij})_{i,j=1,\ldots, n}$,  with  
\begin{equation}\label{matrix_A}
A_{\rm hom}^{ij} = \dashint_{Y^\ast} A(y)\left(\delta_{ij} + \frac{\partial \omega^j}{\partial {y_i}} \right) dy.
\end{equation}
For any $\psi \in C_0(G_T, C_{\rm per}(\Gamma))$, with $\psi(t,x,y) \geq 0$ in $(0,T)\times G\times \Gamma$, using the non-negativity and   two-scale convergence of $u^\ve$ on  $\Gamma^\ve$,  we obtain 
$$
\begin{aligned} 
0\leq  \lim\limits_{\ve \to 0} \ve \langle u^\ve(t,x), \psi(t,x,x/\ve)\rangle_{\Gamma^\ve_T} = |Y|^{-1}\langle u(t,x), \psi(t,x,y)\rangle_{G_T\times \Gamma} = \langle u(t,x), \overline \psi(t,x)\rangle_{G_T},  
\end{aligned}
$$
where 
$$
\overline \psi(t,x) = \frac 1 {|Y|} \int_\Gamma \psi(t,x,y) d\gamma_y \geq 0 \quad \text{ in } \; (0,T)\times G.
$$ 
Hence  $u(t,x) \geq 0$  in $(0,T)\times G$.  The weak convergence in $L^2(0,T; H^1(G))$ of the extension  $\overline u^\ve$ of $u^\ve$,  see Remark~\ref{remark_extension},  ensures that $u(t,x) = \kappa_D$ on $(0,T)\times \partial G$. Thus we have that  $u(t) \in \mathcal K$ for $t\in [0, T]$. 

 Considering $\phi = v-u$,   for any $v\in \kappa_D+ L^2(0,T; H^1_0(G))$ with  $v(t,x) \geq 0$ in $(0,T)\times G$,    as a test function  in \eqref{macro_ineq}     yields  the macroscopic variational inequality \eqref{macro_main}.

 The proof of the uniqueness result  follows the same steps as  the proof of the uniqueness result for the regularised problem \eqref{main2_regul}  in Lemma~\ref{Lemma_existence_regular}.
\end{proof}

\noindent {\bf Remark.} Notice that if  in pseudoparabolic and elliptic terms we  have two  different functions depending on micro\-scopic variables  $y$, i.e.\  $A(y) k(u) \nabla \partial_t u$ and $B(y) k(u) P_c(u) \nabla u$, with $0 < a_0 \leq A(y)\leq A_0 < \infty$ and $0<b_0 \leq B(y)\leq B_0< \infty$, we  need to consider  a modified form for   function $w$, i.e.\ 
\begin{equation}\label{w_new}
\begin{aligned}
w(t,x,y)= \sum_{j=1}^n \frac{\partial u(t,x)}{\partial x_j}   \vartheta^j (y)  + \sum_{j=1}^n \int_0^t \frac{ \partial^2 u(s,x)}{\partial s\partial x_j}   \chi^j(t-s, x, y)  ds  + \overline w(t,x), 
\end{aligned}
\end{equation}
instead of \eqref{form_w},  where $\vartheta^j$ and $\chi^j$ satisfy  the following `unit cell' problems:
\begin{equation}\label{unit_omega_2}
\begin{aligned} 
& {\rm div}_y ( B(y) (\nabla_y \vartheta^j + e_j)) = 0 \quad &&  \text{ in } Y^\ast,  && \int_{Y^\ast} \vartheta^j(y) dy =0, \\
&  B(y) (\nabla_y \vartheta^j + e_j)\cdot \nu = 0  &&\text{ on } \Gamma, \quad && \vartheta^j \;  \;  \; Y-\text{periodic},  
\end{aligned} 
\end{equation}
and  
\begin{eqnarray} \label{macro_unit22}
\begin{aligned}
  {\rm div}_y \left( k(u(t+s))\left[A(y)\nabla_y \partial_t  \chi^j + B(y) P_c(u(t+s)) \nabla_y \chi^j  \right ]\right)  &= 0  \quad  \;\;   \text{in } Y^\ast_{T-s},\;    \\
  k(u(t+s)) [A(y)\nabla_y \partial_t \chi^j  + B(y)P_c(u(t+s)) \nabla_y \chi^j ] \cdot \nu &= 0 \quad \;  \; \text{on } \Gamma_{T-s}, \;     \\
 \chi^j   \qquad   & \qquad \;  \quad  Y-\text{periodic}, \\
  \chi^j(0, x,y) = \omega^j(y) - \vartheta^j(y)  &\quad \;  \qquad   \text{in } Y^\ast, \; \; \;   \int_{Y^\ast} \chi^j (t,x,y) dy = 0, 
  \end{aligned}
\end{eqnarray} 
for $s\in [0,T)$, $x\in G$,  and $j=1,\ldots, n$,  with  $\omega^j$ satisfying  \eqref{macro_unit1}.

The well-posedness of  \eqref{macro_unit1} and \eqref{unit_omega_2} follows from the strict positivity and boundedness of functions $A$ and $B$. 
 To show the well-posedness of  \eqref{macro_unit22}  we first consider the regularised problem  
\begin{equation}\label{macro_unit2_reg}
\begin{aligned} 
&  {\rm div}_y \left(  k(u+\delta)\left[A(y) \nabla_y \partial_t  \chi^j_\delta + B(y) P_c(u+\delta) \nabla_y \chi^j_\delta  \right ]\right) = 0 \;   && \text{in } Y^\ast_{T-s}, \\
&   k(u+\delta) [A(y)\nabla_y \partial_t \chi^j_\delta  + B(y)  P_c(u+\delta) \nabla_y \chi^j_\delta ] \cdot \nu = 0 && \text{on } \Gamma_{T-s}, \\
& \chi^j_\delta               &&   Y-\text{periodic}, \\
&  \chi^j_\delta(0,x,y) = \omega^j(y) - \vartheta^j(y)  \;  && \text{in } Y^\ast, \; \; \int_{Y^\ast} \chi_\delta^j (t,x,y) dy = 0 .
\end{aligned} 
\end{equation}  

\begin{lemma} \label{exist_unit_cell}
Under assumptions on $A$ and $B$ and on nonlinear functions $k$ and $P_c$, see Assumption~\ref{assum_1},  there exists a unique solution $\chi^j \in L^\infty((0,T-s)\times G; H^1_{\rm per}(Y^\ast))$ of \eqref{macro_unit22},  with  $\sqrt{k(u)} \, \partial_t \chi^j \in L^2((0,T-s)\times G; H^1_{\rm per}(Y^\ast))$, for each $j=1, \ldots, n$ and $s\in [0, T)$.  
\end{lemma} 
\begin{proof} 
First we consider the regularised problem   \eqref{macro_unit2_reg}.  
To show existence of a solution of \eqref{macro_unit2_reg} we consider the discretisation in time of   \eqref{macro_unit2_reg}  and  obtain  
\begin{equation}\label{macro_unit2_reg_discrete}
\begin{aligned} 
&  {\rm div}_y \left(  k(u(t_m+s)+\delta)\left[ A(y) \frac 1 h \nabla_y (\chi^j_{\delta, m} - \chi^j_{\delta, m-1}) + B(y) P_c(u(t_m+s)+\delta) \nabla_y \chi^j_{\delta, m}  \right ]\right) = 0 \;   && \text{in } Y^\ast, \\
&   k(u(t_m+s)+\delta) [A(y)\frac 1 h \nabla_y  (\chi^j_{\delta, m}  - \chi^j_{\delta, m-1})   +  B(y) P_c(u(t_m+s)+\delta)  \nabla_y \chi^j_{\delta, m} ] \cdot \nu = 0 && \text{on } \Gamma, \\
& \int_{Y^\ast} \chi^j_{\delta, m} (x, y) dy =0, \hspace{3 cm }   \chi^j_{\delta, m}      \qquad             Y-\text{periodic}, 
\end{aligned} 
\end{equation} 
where $\chi^j_{\delta, 0}(x,y) = \omega^j(y) - \vartheta^j(y)$ in  $Y^\ast$, with $\chi^j_{\delta, 0}(x, \cdot)  \in H$ for $x \in G$, and   $t_m = hm$ for  $h=(T-s)/N$, $s\in [0,T)$,   $m=1, \ldots, N$,  and  $N\in \mathbb N$.   Here  $H = \{ v \in  H^1_{\rm per}(Y^\ast) \, : \; \int_{Y^\ast} v(y) dy = 0 \}$. 

A weak solution of  problem \eqref{macro_unit2_reg_discrete} is a function  $ \chi^j_{\delta, m} \in H$ satisfying  
\begin{equation}\label{discrete_unit_cell_22}
\begin{aligned}   
\Big\la k(u(t_m+s)+\delta)\left[ A(y) \frac 1 h \nabla_y \chi^j_{\delta, m}  + B(y) P_c(u(t_m+s)+\delta) \nabla_y \chi^j_{\delta, m}  \right ], \nabla_y \varphi \Big\ra_{Y^\ast} \\
= 
\frac 1 h \big\la  k(u(t_m+s)+\delta) A(y) \nabla_y  \chi^j_{\delta, m-1}, \nabla_y \varphi \big\ra_{Y^\ast} 
\end{aligned} 
\end{equation}
 for  $x\in G$, $\varphi \in H^1_{\rm per}(Y^\ast)$,  and a given $ \chi^j_{\delta, m-1} \in H$.   Assumptions on  $A$, $B$, $k$, and $P_c$  ensure that problem \eqref{macro_unit2_reg_discrete} is uniformly elliptic and the bilinear map $a: H \times  H \to \mathbb R$ defined as 
$$
a( \chi^j_{\delta, m} , \varphi) = \int_{Y^\ast} k(u(t_m+s)+\delta)\left[ A(y) \frac 1 h \nabla_y \chi^j_{\delta, m}  + B(y) P_c(u(t_m+s)+\delta) \nabla_y \chi^j_{\delta, m}  \right ] \nabla_y \varphi  \, dy
$$
 is coercive and bounded,  $F\in (H^1_{\rm per}(Y^\ast))^\prime$ given by 
$$
\la F, \varphi\ra_{ (H^1_{\rm per}(Y^\ast))^\prime, H^1_{\rm per}(Y^\ast)} = \frac 1 h \int_{Y^\ast}    k(u(t_m+s)+\delta)  A(y) \nabla_y  \chi^j_{\delta, m-1} \nabla_y \varphi   dy
$$
is bounded,   and $\la F , 1 \ra_{ (H^1_{\rm per}(Y^\ast))^\prime, H^1_{\rm per}(Y^\ast)} = 0$. Thus  applying the Lax-Milgram theorem  yields existence of a  unique solution $\chi^j_{\delta, m} \in H$ of \eqref{macro_unit2_reg_discrete} for $x\in G$ and $s\in [0, T)$.

Considering first $ \chi^j_{\delta, m}  -  \chi^j_{\delta, m-1}$ and then  $\chi^j_{\delta, m}$  as  test functions in \eqref{discrete_unit_cell_22},  summing over $m=1, \ldots, l$, for  $1<l\leq N$,  and using assumptions on functions $A$, $B$ $k$, and $P_c$ yield the following a priori estimates 
$$ 
\begin{aligned}
\sum_{m=1}^l h\Big \|  \frac {\nabla_y(\chi^j_{\delta, m} -\chi^j_{\delta, m-1})  } h \Big  \|^2_{L^2(Y^\ast)} +  \sum_{m=1}^l h\Big \| \nabla_y  \chi^j_{\delta, m}  \Big  \|^2_{L^2(Y^\ast)} \leq C
\end{aligned} 
$$
for $x\in G$.  
Here  we used discrete Gronwall and H\"older inequalities and the fact that
$$
\sum_{m=1}^l h \| \nabla_y  \chi^j_{\delta, m}   \|^2_{L^2(Y^\ast)} \leq   C 
\sum_{m=1}^l h  \sum_{i=1}^m h \Big \| \frac{ \nabla_y  (\chi^j_{\delta, i} - \chi^j_{\delta, i-1})} h   \Big  \|^2_{L^2(Y^\ast)}
+ \|\nabla \chi^j_{\delta, 0}\|^2_{L^2(Y^\ast)} . 
$$
Then for  piecewise linear and piecewise constant interpolations given by 
$$
\begin{aligned}
&\hat \chi^j_{\delta, N} (t,x,y)= \chi^j_{\delta, m-1}(x,y) + (t-t_{m-1}) \frac{\chi^j_{\delta, m}(x,y) - \chi^j_{\delta, m-1}(x,y)} h \; \text{ for } \;    t \in (t_{m-1}, t_m],\\
& \bar \chi^j_{\delta, N} (t,x,y) = \chi^j_{\delta, m}(x,y) \; \; \text{ for } \;   t \in (t_{m-1}, t_m], \; \; \;    m=1, \ldots, N, 
\end{aligned}
$$
 for $x\in G$ and $y \in Y^\ast$,  using the zero-mean value of $\chi^j_{\delta, m}$ and the Poincar\'e inequality,  we obtain  
$$
\| \partial_t  \hat \chi^j_{\delta, N} \|_{L^2(Y^\ast_{T-s})}+ \| \partial_t \nabla_y \hat \chi^j_{\delta, N} \|_{L^2(Y^\ast_{T-s})} +
 \|  \bar \chi^j_{\delta, N} \|_{L^2(Y^\ast_{T-s})}    + \| \nabla_y \bar \chi^j_{\delta, N} \|_{L^2(Y^\ast_{T-s})} \leq C, 
$$
for $x\in G$ and  a constant $C$ independent of $N$ and $x\in G$.  Last estimates ensure that there exists a function $\chi^j_\delta \in H$,   with $\partial_t \chi^j_\delta \in H$, such that 
$$
\begin{aligned} 
\bar \chi^j_{\delta, N} \rightharpoonup \chi_\delta^j \quad \text{ weakly$^\ast$  in } \; L^2(0, T-s; L^\infty(G; H^1(Y^\ast))), \\
\partial_t \hat \chi^j_{\delta, N} \rightharpoonup \partial_t \chi_\delta^j \quad \text{ weakly$^\ast$  in } \; L^2(0, T-s; L^\infty(G; H^1(Y^\ast))), 
\end{aligned} 
$$
as $N \to \infty$.  
Using continuity of $u$ with respect to time variable,  integrating  \eqref{discrete_unit_cell_22} with respect to $t$ and $x$, and taking the limit as $N \to \infty$ yield  that $\chi^j_\delta$ is a weak solution of the regularised `unit cell' problem \eqref{macro_unit2_reg}.  The linearity of the problem  and properties of $A$, $B$, $k$,  and $P_c$ ensure the uniqueness of a solution of \eqref{macro_unit2_reg}. 

Now we shall derive a priori  estimates for   $\chi^j_\delta$, uniformly in $\delta$.  Considering  $\chi^j_\delta / k(u+\delta)$ and $\partial_t \chi^j_\delta $  as test functions in the weak formulation of     \eqref{macro_unit2_reg} we obtain 
\begin{equation}\label{estim_unit_cell}
\begin{aligned} 
\|\nabla_y \chi_\delta^j\|_{L^\infty(0,T-s; L^2(Y^\ast))} +  \|\sqrt{P_c (u+ \delta)} \nabla_y \chi_\delta^j \|_{L^2((0,T-s)\times Y^\ast)} 
+ 
 \|\sqrt{k (u+ \delta)} \nabla_y\partial_t \chi_\delta^j \|_{L^2((0,T-s)\times Y^\ast)}  \leq C, 
\end{aligned} 
\end{equation} 
for  $x\in G$ and a constant $C>0$ independent of $\delta$ and $x\in G$. Assumptions on  $k$ and $P_c$ in Assumption~\ref{assum_1}, together with the additional assumption that $k$ is continuously differentiable for $z\geq 0$,   combined with  the regularity $\partial_t u \in L^{p}(0,T; L^{q_1}(G))$, where   $q_1=pn/(n-p)$ and $1<p<2$,  imply 
$$
\begin{aligned}
\langle  k(u+\delta) \nabla_y \partial_t  \chi^j_\delta, \nabla_y \psi \rangle_{Y^\ast_{T-s}} =  - \langle  k^\prime(u+\delta) \partial_t u \nabla_y  \chi^j_\delta, \nabla_y \psi \rangle_{Y^\ast_{T-s}} - \langle  k(u+\delta) \nabla_y  \chi^j_\delta, \nabla_y \partial_t  \psi \rangle_{Y^\ast_{T-s}}, 
 \end{aligned}
$$
for $\psi \in C^1_0(G_{T-s} \times Y^\ast)$ and $x\in G$. 
Taking in the last equality the limit as $\delta \to 0$ and considering   estimates  in \eqref{estim_unit_cell} yield 
$$
\sqrt{k(u+\delta)} \nabla_y \partial_t  \chi^j_\delta  \rightharpoonup   \sqrt{k(u)} \nabla_y \partial_t  \chi^j \qquad \text{weakly  in }\;  L^2(G_{T-s}\times Y^\ast).
$$
Then, using  the continuity of  $k$  and $P_c$,  regularity of $\partial_t u$ and the  estimate for $\nabla_y \chi^j_\delta$ in $L^\infty(G_{T-s};  L^2( Y^\ast))$,   we can pass to the limit as $\delta \to 0$ in the weak formulation of  \eqref{macro_unit2_reg} and obtain that the limit function $\chi^j \in L^\infty(G_{T-s}; H^1_{\rm per}(Y^\ast))$, with $\sqrt{k(u)} \partial_t \chi^j \in L^2(G_{T-s}; H^1_{\rm per}(Y^\ast))$, is a solution of~\eqref{macro_unit22}. 

To prove the uniqueness result for \eqref{macro_unit22} we assume that there are two solutions $\chi^j_1$ and $\chi^j_2$ of \eqref{macro_unit22} and  consider $ [k(u+ \delta)]^{-1} (\chi^j_1 - \chi^j_2)$ as a test function in the weak formulation of the equations for the difference  $(\chi^j_1- \chi^j_2)$  to obtain 
$$
\int_{G_{\tau-s}} \int_{Y^\ast}  \frac{k(u(t+s))}{k(u(t+s)+\delta)} \Big( A(y) \partial_t \nabla_y (\chi^j_1 - \chi^j_2) \nabla_y (\chi^j_1 - \chi^j_2) + B(y) P_c(u(t+s)) |\nabla_y (\chi^j_1 - \chi^j_2)|^2 \Big)dy dx dt= 0, 
$$
for $\tau \in (s, T]$. 
 Integrating by parts in the first term of the last equality, using assumptions on  $A$, $B$,  $k$, and $P_c$, together with  nonnegativity of $u$,   and taking limit as $\delta \to 0 $  imply  
$$
\sup_{(0, T-s)} \|\nabla_y (\chi^j_1 - \chi^j_2)\|_{L^2(G\times Y^\ast)} \leq 0 . 
$$
Then  Poincar\'e inequality and  the fact that the mean value of $\chi^j_l$, for $l=1,2$, is zero ensure   $\chi^j_1 = \chi^j_2$ a.e. in $G\times Y^\ast_{T-s}$,  for $s\in [0, T)$ and $j=1, \ldots, n$.  
\end{proof} 

Considering  the expression   \eqref{w_new} for $w$ in  \eqref{macro_ineq}  and choosing $\phi = v - u$ yield  the corresponding macroscopic variational inequality 
\begin{equation*}
\begin{aligned} 
& \langle \partial_t b(u), v - u \rangle_{G_T}   +\big \langle  k(u)  [ A_{\rm hom} \partial_t \nabla u  +  B_{\rm hom}   P_c(u)  \nabla u  ] , \nabla(v - u) \big\rangle_{G_T} \\ 
&+\Big\langle \int_0^t K_{\rm hom}(t-s,x) \partial_s \nabla u \, ds ,  \nabla(v - u) \Big\rangle_{G_T} 
  - \langle   F_{\rm hom}(t,x,u) ,  \nabla (v - u) \rangle_{G_T}   + 
\langle  f_{\rm hom}(t,u), v - u \rangle_{G_T}  \geq 0, 
\end{aligned} 
\end{equation*} 
for $v\in L^2(0, T; \mathcal K)$, where $A_{\rm hom}$, $F_{\rm hom}$ and $f_{\rm hom}$ are defined  in Theorem~\ref{th:macro}, and matrices  $B_{\rm hom}= ( B_{\rm hom}^{ij})$ and   $K_{\rm hom}(t,x)= ( K_{\rm hom}^{ij}(t,x))$ are determined by 
$$
\begin{aligned}
&B_{\rm hom}^{ij} = \dashint_{Y^\ast} B(y)\left(\delta_{ij} + \frac{\partial \vartheta^j}{\partial {y_i}} \right) dy, \\
&K_{\rm hom}^{ij}(t,x) = \dashint_{Y^\ast} k(u(t+s,x))[ A(y)  \partial_t \partial_{y_i} \chi^j  + B(y)  P_c(u(t+s,x)) \partial_{y_i} \chi^j ]  dy. 
\end{aligned}
$$

\section*{Appendix}
 \begin{definition}\cite{Allaire, Nguetseng}
A sequence    $\{u^{\varepsilon}\} \subset L^p(G)$  converges two-scale   to $u$,  with 
 $ u \in L^p(G \times Y)$,  iff for any $\phi \in L^q(G; C_{\rm per}(Y))$ we have 
 $$
\lim_{\varepsilon \rightarrow 0}\int_{G}u^{\varepsilon}(x)\phi\left(x,\frac{x}{\varepsilon}\right)dx 
= \int_{G}\dashint_{ Y} u(x,y)\phi(x,y)dx dy, 
$$
where  $1/p+1/q=1$. 
\end{definition}

\begin{definition}\cite{Allaire, Neuss}
 A sequence $\{u^{\varepsilon}\}\subset L^2({\Gamma}^{\ve})$  converges  two-scale to  $u$, with  $u \in L^2(G \times \Gamma)$, iff  for
$\psi \in C_0(G; L_{\rm per}^2(\Gamma))$ holds
$$
\lim_{\varepsilon \rightarrow 0} \varepsilon  \int_{{\Gamma}^{\varepsilon}}u^{\epsilon}(x)\psi\Big(x,\frac x{\varepsilon}\Big)\, d\gamma_x =
\frac 1{|Y|}\int_{G}\int_{\Gamma} u(x,y)\psi(x,y)\, dx d\gamma_y .
$$
\end{definition}

\begin{theorem}[Compactness]\cite{Allaire, Nguetseng}
 Let $\{u^{\varepsilon}\}$  be a bounded sequence in  $H^{1}(G)$, which converges weakly to  $u \in H^{1}(G)$. Then there exists $u_1 \in L^2(G; H^{1}_{\rm per}(Y))$ such that, up to a subsequence, $u^{\varepsilon}$ two-scale converges to $u$ and $\nabla u^{\varepsilon}$ two-scale converges to $\nabla u+\nabla_y u_1$.\\
Let $\{u_{\varepsilon}\}$ and $\{{\varepsilon} \nabla u^{\varepsilon}\}$  be  bounded sequences in  $L^2(G)$. Then there exists $u_0 \in L^2(G; H^{1}_{\rm per}(Y))$ such that, up to a subsequence, $u^{\varepsilon}$ and ${\varepsilon} \nabla u^{\varepsilon}$ two-scale converge to $u_0$ and $\nabla_y u_0$, respectively.\\
Let $\{\sqrt{\ve} u_{\varepsilon}\}$  be a  bounded sequences in  $L^2(\Gamma^\ve)$. Then there exists $u_0 \in L^2(G \times \Gamma)$ such that, up to a subsequence, $u^{\varepsilon}$  two-scale converge to $u_0$. 
\end{theorem}


\begin{thebibliography}{100}
\bibitem{Acerbi}
Acerbi E.,  Chiado Piat V., Dal Maso G.,  Percivale D. {An extension theorem from connected sets, and homogenization in general periodic domains}. {\it Nonlinear Anal.~Theory Methods Appl.},  1992, 18, 481--496.

\bibitem{Allaire}
Allaire G. Homogenization and two-scale convergence. {\it SIAM J.~Math.~Anal.}, 1992, 23, 1482--1518.

\bibitem{Barenblatt_1990}
Barenblatt G., Entov V., Ryzhik V.  {\it Theory of fluid flow through natural rocks}. Dordrecht: Kluwer, 1990.

\bibitem{Barenblatt_1993}
Barenblatt, G., Bertsch M., Passo R.D.,  Ughii M. A degenerate pseudoparabolic regularization of a nonlinear forward-backward heat equation arising in the theory of heat and mass exchange in stably stratified turbulent shear flow. {\it SIAM J Math.~Anal.},  1993, 24, 1414--1439.

\bibitem{Behrndt_Micheler}
Behrndt J., Micheler T. Elliptic differential operators on Lipschitz domains and abstract boundary value problems. {\it J Func. Anal.},  2014, 267, 3657--3709. 

\bibitem{Boehm_1985}
Boehm M.,  Showalter R.E.  A nonlinear pseudoparabolic diffusion equation.
{\it SIAM J.~Math.~Anal.}, 1985, 16, 980--999.

\bibitem{Brown_Shen}
Brown R.M., Shen Z. Estimates for the Stokes operator in Lipschitz domains. {\it Indiana Univ. Math. J.}, 1995,  44, 1183--1206. 

\bibitem{Pop}
Canc\'es C., Choquet C.,  Fan Y.,  Pop I.S.  Existence of weak solutions to a degenerate pseudo-parabolic equation modelling two-phase flow in porous media.  {\it Technical report, Eindhoven University of Technology, CASA-Report}, 2010, 10--75.

\bibitem{Timofte_2016}
Capatina A., Timofte C. Homogenization results for micro-contact elasticity problems. {\it J.~Math.~Anal.~Appl.},  2016, 441,  462--474.

\bibitem{CiorPaulin99}
Cioranescu D., Saint Jean Paulin J. {\it Homogenization of reticulated structures}.  Springer, New York, 1999.

\bibitem{Cioranescu}
Cioranescu D.,  Damlamian A.,  Donato P.,   Griso G.,  Zaki R. The periodic unfolding
method in domains with holes. {\it SIAM J Math. Anal.},  2012, 44, 718--760.

\bibitem{Cuesta}
Cuesta C., van Duijn C.J., Hulshof J.  Infiltration in porous media with dynamic capillary pressure: Travelling waves. {\it Euro.~Jnl Appl.~Math.},  2000, 11,  381--397.

\bibitem{DiBenedetto}
DiBenedetto E., Showalter R.E.  A pseudo-parabolic variational inequality and Stefan problem. {\it Nonlinear Anal.~Theory Methods Appl.}, 1982, 6, 279--291.

\bibitem{Iosifyan}
Iosifyan G.A. Homogenization of elastic problems with boundary conditions of Signorini type. {\it Math.~Notes},  2004, 75, 765--779.

\bibitem{Geng_Kilty}
Geng J.,  Kilty J.  The Lp regularity problem for the Stokes system on Lipschitz domains. {\it  J. Differential Equations},  2015, 113, 53--172. 
259,  1275--1296.

\bibitem{Gesztesy_Mitrea} 
Gesztesy F., Mitrea M.  A description of all self-adjoint extensions of the Laplacian and Kre\v i n-type resolvent formulas on non-smooth domains. 
{\it J. d'Analyse Math.}, 2011, 
 
\bibitem{Hassanizadeh}
Hassanizadeh S.M.,  Gray W.G. Thermodynamic basis of capillary pressure in porous media. {\it Water Resour.~Res.}, 1993, 29, 3389--3405.

\bibitem{Hornung}
Hornung U. {\it Homogenization and porous media}. Springer-Verlag, 1997.

\bibitem{Jaeger}
Hornung U., J\"ager W., Mikelic A.  Reactive transport through an array of cells with semi-permeable membranes.  {\it RAIRO - Mod\'el. Math. Anal. Num\'er.}, 1994, 28, 59--94. 

\bibitem{Jaeger_2014}
J\"ager W.,  Neuss-Radu M.,  Shaposhnikova T.A. Homogenization of a variational inequality for the Laplace operator with nonlinear restriction for the flux on the interior boundary of a perforated domain. {\it Nonlinear Analysis: Real World Appl.}, 2014, 15, 367--380.

\bibitem{Kenig}
Kenig C.E., Lin F., Shen Z. Homogenization of elliptic systems with Neumann boundary conditions. {\it J Amer. Math. Soc.}, 2013, 26, 901--937. 

\bibitem{Kenneth_1984}
Kenneth L., Kuttler J. Degenerate variational inequalities of evolution. {\it Nonlinear Anal.~Theory~Methods~Appl.},  1984, 8, 837--850.

\bibitem{Kinderlehrer}
 Kinderlehrer D.,  Stampacchia S. {\it An introduction to variational inequalities and their applications.}  Academic Press, New York, 1980. 

\bibitem{Lions}
Lions  J.L. {\it Quelques m\'ethodes de r\'esolution des probl\`emes aux limites non lin\'eaires.}  Dunod, Paris, 1969.

\bibitem{Ptashnyk_3}
Marciniak-Czochra, A., Ptashnyk, M.  Derivation of a macroscopic receptor-based model using homogenization techniques. 
{\it SIAM J.~Math.~Anal.}, 2008,  40, 215-237.

\bibitem{Melnyk_2011}
Mel'nyk T.A., Nakvasiuk Iu.A., Wendland, W.L. Homogenization of the Signorini boundary-value problem in a thick junction and boundary integral equations for the homogenized problem. {\it Math. Methods Appl.~Scie},  2011, 34, 758--775. 

\bibitem{Melnyk_2016}
Mel'nyk T.A., Nakvasiuk Iu.A. Homogenization of a semilinear variational inequality in a thick multi-level junction. {\it J Inequal.~Appl.},  2016, 104, 1--22.

\bibitem{Melnyk_2012}
Mel'nyk T.A.,  Nakvasiuk Iu.A. Homogenization of a parabolic Signorini boundary-value problem in a thick plane junction. {\it J Math.~Scie},  2012, 181, 613--631. 

\bibitem{Mielke_Timofte}
Mielke A., Timofte C.   Two-scale homogenization for evolutionary variational inequalities via the energetic formulation. 
{\it   SIAM J.~Math.~Anal.},  2007, 39,  642--668.

\bibitem{Mikelic}
Mikeli\'c A.   A global existence result for the equations describing unsaturated flow in porous media with dynamic capillary pressure. {\it J.~Differ.~Equat.}, 2010, 248, 1561--1577.

\bibitem{Mikelic_1991}
Mikeli\'c A.  Homogenization of nonstationary Navier-Stokes equations in a domain with a grained boundary. {\it Annal. Matem. Pure Appl.}, 19991,  CLVIII, 167--179. 


\bibitem{Mitrea_Wright}
Mitrea M., Wright M. Boundary value problems for the Stokes system in arbitrary Lipschitz domains. {\it  Ast\'erisque}, 2012,  344, viii+241. 

\bibitem{Murat_1997}
Murat F., Tartar L. {\it H-convergence. in Topics in the mathematical modelling of composite materials}, 21--43, Progr. Nonlinear Differential Equations Appl. 31, Birkh\"auser Boston, Boston, MA, 1997.

\bibitem{Neuss}
Neuss-Radu M. Some extensions of two-scale convergence. {\it C. R. Acad. Sci. Paris S\'er. I Math.}, 1996, 332, 899--904.


\bibitem{Nguetseng}
Nguetseng G. A general convergence result for a functional related to the theory of homogenization, {\it SIAM J.~Math.~Anal.},  1989, 20,  608--623. 

\bibitem{Novick}
Novick-Cohen A.,  Pego R.L.  Stable patterns in a viscous diffusion equation. {\it Trans.~Amer.~Math.~Soc.}, 1991, 324, 331--351.


\bibitem{Pastukhova_2001}
Pastukhova S.E. Homogenization of a mixed problem with Signorini condition for an elliptic
operator in a perforated domain. {\it Sb. Math.}, 2001, 192, 245--260. 

\bibitem{Pavone_1989}
Pavone D. Macroscopic equations derived from space averaging for immiscible two-phase flow in porous media. {\it Oil \& Gas Science and Technology - Rev. IFP}, 1989,  44, 29--41.

\bibitem{Peszynska}
Peszy\'nska M., Showalter R., Yi S.-Y.
Homogenization of a pseudoparabolic system. {\it Applicable Analysis},  2009, 88, 1265--1282. 


\bibitem{Ptashnyk_1}
Ptashnyk M.  Degenerate quasilinear pseudoparabolic equations with memory terms and variational inequalities.
{\it Nonlinear Anal.~Theory Methods Appl.}, 2006, 66, 2653--2675. 

\bibitem{Ptashnyk_2}
Ptashnyk M.   Nonlinear pseudoparabolic equations as singular limit of reaction-diffusion equations. {\it Applicable~Analysis}, 2006, 85, 1285--1299. 



\bibitem{Rodrigues_1982}
Rodrigues J-F.   Free boundary convergence in the homogenization of the one phase Stefan problem. {\it Trans. Amer. Math, Society}, 1982, 274, 297--3002. 

\bibitem{Rubinshtein}
Rubinshtein L.I. On the process of heat transfer in heterogeneous media. {\it lzv. Akad. Nauk SSSR. Set. Geogrqf. Geofiz.}, 1948,  12, 27-45.

\bibitem{Sandrakov}
Sandrakov G.V. Homogenization of variational inequalities and equations de ned by pseudomonotone operators. {\it Sbornik: Mathematics}, 2008,  199,  67--98. 

\bibitem{Scarpini_1987}
Scarpini, F.  Degenerate and pseudoparabolic variational inequalities. {\it Numer.~Funct.~Anal.~Optim.}, 1987,  9, 859--879.

\bibitem{Shaposhnikova_2008}
Shaposhnikova T.A.,  Zubova M.N. Homogenization problem for a parabolic variational inequality with constraints on subsets situated on the boundary of the domain. 
{\it Networks \& Heterogeneous Media}, 2008, 3, 1--20.


\bibitem{Showalter}
Showalter R.E., Ting T.W. Pseudoparabolic partial differential equations, {\it SIAM J.~Math.~Anal.},  1970,  1, 1--26.


\bibitem{Showalter_1996}
Showalter, R.E.  {\it Monotone Operators in Banach Space and Nonlinear Partial Differential Equations}. Mathematical Surveys and Monographs, 1996 

\bibitem{Tartar}
Tartar L. Incompressible fluid flow in a porous medium - convergence of the homogenization process. {\it Appendix in Lecture Notes in Physics} 127, Springer, Berlin, 1980

\bibitem{Vorobev_2003}
Vorob'ev A.Yu., Shaposhnikova T.A. Homogenizaton of a nonhomogeneous Signorini problem for the Poisson equation in a periodically perforated domain. 
{\it Differential  Equations},  2003, 39, 387--396.

\bibitem{Xu_2017}
Xu Q. Convergence rates and $W^{1,p}$ estimates in homogenization theory of Stokes systems in Lipschitz domains. {\it J Differ. Equat.}, 2017, 263, 398--450. 
\end{thebibliography}
\end{document}